\documentclass[letterpaper, 10 pt, conference]{ieeeconf}  

\IEEEoverridecommandlockouts 
\usepackage[utf8]{inputenc} 
\usepackage[T1]{fontenc}    
\usepackage{url}            
\usepackage{booktabs}       
\usepackage{nicefrac}       
\usepackage{microtype}      
\usepackage{multirow}
\usepackage{array}
\usepackage{cite}
\usepackage[noblocks]{authblk}
\usepackage{amssymb,arydshln}
\usepackage[nodayofweek]{datetime}

\def\cmtinclude{1}
\if\cmtinclude1
\newcommand{\ys}[1]{\noindent{\textcolor{blue}{\{{\bf YS:} #1\}}}}
\else
\newcommand{\ys}[1]{}
\fi
\let\labelindent\relax
\usepackage{enumitem}
\setlist[itemize]{leftmargin=*}

\usepackage[font=small,labelfont=bf]{caption}
\usepackage{amsmath,amsfonts,nccmath,mathtools,mathrsfs}
\usepackage{tcolorbox}
\usepackage{xcolor}
\usepackage[colorlinks = true,
            linkcolor = blue,
            urlcolor  = blue,
            citecolor = blue,
            anchorcolor = blue]{hyperref}

\usepackage[margin=1in]{geometry}

\usepackage{caption,balance} 
\usepackage{amsthm}
\usepackage{graphicx,color}
\usepackage{subcaption,setspace}

\newtheorem{theorem}{Theorem}

\newtheorem{lemma}{Lemma}
\newtheorem{corollary}{Corollary}

\newtheorem{assumption}{Assumption}

\theoremstyle{definition}

\newtheorem{example}{Example}

\theoremstyle{remark}
\newtheorem{remark}{Remark}

\usepackage[nameinlink]{cleveref}
\crefname{equation}{}{}
\crefname{theorem}{Theorem}{Theorems}
\crefname{corollary}{Corollary}{Corollaries}
\crefname{example}{Example}{Examples}
\crefname{assumption}{Assumption}{Assumptions}
\crefname{lemma}{Lemma}{Lemmas}
\crefname{proposition}{Proposition}{Propositions}
\crefname{figure}{Figure}{Figures}
\crefname{table}{Table}{Tables}
\crefname{section}{Section}{Sections}
\crefname{appendix}{Appendix}{Appendices}
\Crefname{equation}{}{}
\Crefname{theorem}{Theorem}{Theorems}
\Crefname{corollary}{Corollary}{Corollaries}
\Crefname{example}{Example}{Examples}
\Crefname{lemma}{Lemma}{Lemma}
\Crefname{proposition}{Proposition}{Proposition}
\Crefname{figure}{Figure}{Figures}
\Crefname{table}{Table}{Tables}
\Crefname{section}{Section}{Sections}
\Crefname{appendix}{Appendix}{Appendices}

\newcommand{\tr}{{{\mathsf T}}}
\newcommand{\mK}{{\mathsf{K}}}

\usepackage{algorithm}
\usepackage{algorithmic}
\usepackage{amsmath}
\usepackage{amssymb}
 \usepackage{multirow}
 \usepackage{mathrsfs}
\usepackage{amsfonts}

\newcommand{\cC}{{\mathcal C}}

\newcommand{\diag}{{\rm diag}}

\DeclareFontFamily{U}{mathx}{\hyphenchar\font45}
\DeclareFontShape{U}{mathx}{m}{n}{
      <5> <6> <7> <8> <9> <10>
      <10.95> <12> <14.4> <17.28> <20.74> <24.88>
      mathx10
      }{}
\DeclareSymbolFont{mathx}{U}{mathx}{m}{n}
\DeclareMathAccent{\widecheck}{\mathalpha}{mathx}{"71}
\newcommand{\R}{{\mathbb{R}}}

\newcommand{\Ib}{I}
\newcommand{\cH}{{\mathcal H}}

\newcommand{\tCK}{\tilde C_K}
\newcommand{\tAK}{\tilde A_K}
\newcommand{\tBK}{\tilde B_K}
\newcommand{\qp}{q^\prime}

\newcommand{\gth}{g_{\text{th}}}
\newcommand{\tth}{t_{\text{th}}}

\newcommand{\K}{\mK}

\newcommand{\lbmin}{\lambda_{\text{Han},\min}}
\newcommand{\lbminK}{\lbmin(\K_t)} 
\newcommand{\redord}{\text{reduce\_order}}

\makeatletter
\newcommand{\removelatexerror}{\let\@latex@error\@gobble}
\makeatother

\allowdisplaybreaks
\usepackage{setspace}
\title{\bf Escaping High-order Saddles in Policy Optimization for Linear Quadratic Gaussian (LQG) Control}

\author{Yang Zheng$^{\dag1}$ \quad Yue Sun$^{\dag2}$ \quad Maryam Fazel$^{3}$ \quad Na Li$^{4}$
\thanks{$\dag$ Y. Zheng and Y. Sun  contributed to this work equally. The research of M. Fazel and Y. Sun was supported in part by awards NSF TRIPODS II 2023166, CCF 1839291, and CCF 2007036. The work of N. Li was supported by NSF CAREER  ECCS-1553407, NSF AI institute 2112085, and ONR YIP N00014-19-1-2217.}
\thanks{$^{1}$ Y. Zheng ({zhengy@eng.ucsd.edu}) is with ECE Department, University of California San Diego, La Jolla, USA.}
\thanks{$^{2}$ Y. Sun ({yuesun9308@gmail.com}) worked on this project while he was with ECE Department, University of Washington, Seattle, USA.}
\thanks{$^{3}$ M. Fazel ({mfazel@uw.edu}) is with ECE Department, University of Washington, Seattle, USA.} 
\thanks{$^{4}$ N. Li ({nali@seas.harvard.edu}) is with Department of Electrical Engineering and Applied Mathematics, Harvard University, USA.}
}

\begin{document}
\maketitle

\begin{abstract}
    First order policy optimization has been widely used in reinforcement learning. It guarantees to find the optimal policy for the state-feedback linear quadratic regulator (LQR). However, 
    the performance of policy optimization remains unclear for the linear quadratic Gaussian (LQG) control where the LQG cost has spurious suboptimal stationary points. 
    In this paper, we introduce a novel perturbed policy gradient (PGD) method to escape a large class of bad stationary points (including high-order saddles). In particular, based on the specific structure of LQG, we introduce a novel reparameterization procedure which converts the iterate from a high-order saddle to a strict saddle, from which standard random perturbations in PGD can escape efficiently. We further characterize the high-order saddles that can be escaped by our algorithm. 
    
\end{abstract}
\section{Introduction}

In this paper, we revisit the linear quadratic Gaussian (LQG) control, one of the most fundamental problems in control theory, from a modern optimization view. In brief, we focus on a continuous-time linear time-invariant (LTI) system
\begin{equation} \label{eq:LTIdynamics}
\begin{aligned}
    \dot x(t) &= Ax(t) + Bu(t) + w(t),\\
    y(t) &= Cx(t) + 
     v(t), 
\end{aligned}
\end{equation}
where $x(t) \in \mathbb{R}^{n},u(t)\in \mathbb{R}^{m},y(t) \in \mathbb{R}^{p}$ are the state, control {input}, and measurement (output) vector at time $t$, respectively, and $w(t) $,  $v(t)$ are white Gaussian noises with intensity matrices $W \succeq 0$ and $V \succ 0$, respectively. The goal is to design a controller (i.e., policy) based on partial measurements $y(t)$ to minimize a quadratic cost 
\begin{equation} \label{eq:LQGcost}
J(u) := \lim_{T \rightarrow \infty }\frac{1}{T}\mathbb{E} \left[\int_{t=0}^T \left(x^\tr Q x + u^\tr R u\right)dt\right].
\end{equation}
%

A special case is the linear quadratic regulator (LQR)~\cite{kalman1960contributions}, where we have direct access to the state $x$ (i.e., $y(t)=x(t), v(t) = 0, \forall t \in \mathbb{R}$ in \cref{eq:LTIdynamics}). 
It is known that the optimal policy for the LQR is in the form of static state feedback $u(t) = Kx(t)$, where $K \in \mathbb{R}^{m \times n}$ is a constant matrix that can be obtained by solving a Riccati equation~\cite{dullerud2013course}. On the other hand, when the state is not directly observed, the policy that minimizes~\cref{eq:LQGcost} is a dynamical controller of the form~\cite{zhou1996robust}
\begin{equation} \label{eq:dynamical-controller}
\begin{aligned}
    \dot \xi(t) &= A_{\mK} \xi(t) + B_{\mK} y(t),\\
          u(t) &= C_{\mK}\xi(t),
\end{aligned}
\end{equation}
where the optimal parameters $\mK^* := (A_{\mK}^*, B_{\mK}^*, C_{\mK}^*)$ can be obtained by solving two Riccati equations \cite{zhou1996robust,bertsekas2012dynamic} (see \Cref{subsection:review}). 
Algorithms for solving Riccati equations are well-studied, including iterative algorithms~\cite{hewer1971iterative}, algebraic solution methods~\cite{lancaster1995algebraic}, and semidefinite optimization~\cite{balakrishnan2003semidefinite}. All these methods are model-based and explicitly rely on the system model~\cref{eq:LTIdynamics}. 
Recently, policy gradient methods have achieved impressive results for many challenging problems~\cite{recht2019tour}. These methods directly optimize the quadratic cost~\cref{eq:LQGcost} as a function of the policy class $\mK = (A_{\mK}, B_{\mK}, C_{\mK})$ via gradient descent or its variants. They can be further made model-free, bypassing an explicit estimation of the model~\cref{eq:LTIdynamics}. The flexibility of model-free control has stimulated a growing interest in investigating foundations of policy gradient methods for classical control problems~\cite{fazel2018global, mohammadi2019global,malik2019derivative,sun2021learning,furieri2020learning,tu2019gap,zhang2019policy,umenberger2022globally,duan2022optimization,bu2019lqr,hu2022connectivity,zheng2021analysis}.     



While it is guaranteed to obtain the optimal controller for LQR or LQG via classical model-based methods, such optimality guarantee is more difficult when using policy gradient methods since the cost~\cref{eq:LQGcost} is typically nonconvex in the policy space. 
For LQR, recent work has shown that although the LQR cost is nonconvex, it  is gradient dominant and coersive, and has a unique stationary point under very mild conditions, rendering the convergence of policy gradient methods  to the globally optimal controller \cite{fazel2018global, mohammadi2019global,malik2019derivative,sun2021learning}. On the other hand, the LQG cost is neither gradient dominant nor coersive, and there may exist spurious saddle points \cite{zheng2021analysis}, making it challenging for policy gradient to find the optimal controller. 

Saddle points do not always destroy the performance of policy gradient methods. 
Suitable perturbed policy gradient methods are able to escape \textit{strict saddle points} whose Hessian has at least one strictly negative eigenvalue~\cite{ge2015escaping,jin2017escape,carmon2016gradient}. However, it is shown in \cite[Theorem 4.2]{zheng2021analysis} that the Hessian of the LQG cost at a saddle point can even degenerate to zero. We denote the saddle point whose Hessian does not give escaping directions as a \emph{high-order saddle point}. Perturbed policy gradient methods~\cite{ge2015escaping,jin2017escape,carmon2016gradient} may thus get stuck  and take an exponential number of iterations to escape high-order saddle points. 

All the (strict or high-order) saddle points of LQG discussed in \cite{zheng2021analysis} are due to a loss of \textit{controllability} and/or \textit{observability} for the controller $(A_{\mK}, B_{\mK}, C_{\mK})$ in \Cref{eq:dynamical-controller} (i.e., \textit{non-minimal} controllers). Indeed, any stationary point corresponds to a full-order \textit{minimal controller} cannot be saddle and it is instead globally optimal~\cite{zheng2021analysis}. Further, many intrigue landscape properties of LQG are brought by a classical notion of \textit{similarity transformations} that induces a symmetry structure \cite{zheng2021analysis}. In this paper, we raise a natural question of whether this induced symmetry structure allows us to reveal more information about high-order saddles of LQG such that suitable perturbed policy gradient methods can escape those points. We provide a positive answer to this question. 

In particular, we first show that any stationary point after model reduction remains to be stationary. This gives a classification of the stationary points: all bad (suboptimal or saddle) stationary points after model reduction become lower-order and form new stationary points with the same LQG cost.  
We then reveal an intriguing transfer function $\mathbf{G}(s)$ at any stationary point $(A_{\mK}, B_{\mK}, C_{\mK})$: 1) if $(A_{\mK}, B_{\mK}, C_{\mK})$ is globally optimal, the function $\mathbf{G}(s)$ is identically zero, $\forall s \in \mathbb{C}$; 2) if $\mathbf{G}(s)$ is not identically zero, we can perturb $(A_{\mK}, B_{\mK}, C_{\mK})$  to get a new stationary point 
with the same LQG cost, which is a strict saddle with probability one. Standard perturbed policy gradient (PGD) methods~\cite{ge2015escaping,jin2017escape} can thus escape this new strict saddle. We emphasize that our PGD method include perturbations on two parts: 1) a novel structural perturbation on the stationary point $(A_{\mK}, B_{\mK}, C_{\mK})$; 2) a standard random perturbation on gradients~\cite{jin2017escape}. This combination enables escaping a large class of bad stationary points (including high-order saddles) in LQG.  

The rest of this paper is organized as follows. We 
present the problem statement in \Cref{s:LQG review}. Our main results on characterizing stationary points and Hessians are presented in \Cref{section:hessians}. \Cref{section:PGD} shows empirical performance of our perturbed policy gradient method. We conclude the paper in \Cref{section:conclusion}. Technical proofs and auxiliary computations are postponed to the appendix.

\section{Preliminaries and Problem statement}\label{s:LQG review}


\subsection{Review of LQG control} \label{subsection:review}

The classical LQG control problem is defined as 
\begin{equation} \label{eq:LQG}
    \begin{aligned}
        \min_{u(t)} \quad & J(u) \\
        \text{subject to} \quad & ~\cref{eq:LTIdynamics},
    \end{aligned}
\end{equation}
    where $J(u)$ is defined in~\cref{eq:LQGcost}  with $Q \succeq 0$ and $R\succ 0$. In~\cref{eq:LQG}, the input $u(t)$ depends on all past observation $y(\tau)$ with $\tau < t$. We make the following standard assumption.
\begin{assumption} \label{assumption:stabilizability}
$(A,B)$ and $(A,W^{1/2})$ are controllable, and $(C,A)$ and $(Q^{1/2},A)$ are observable.
\end{assumption}


The optimal solution to~\cref{eq:LQG} is a dynamical controller in the form of~\cref{eq:dynamical-controller},  
in which $\xi(t)\in\R^q$ is the controller internal state, and $A_{\mK}\in\R^{q\times q}$, $B_{\mK}\in\R^{q\times p}$, $C_{\mK}\in\R^{m\times q}$ specify the dynamics of the controller. 
%
%
While $q$ can be any positive number, one does not have to use $q>n$ and the optimal controller has $q=n$, given by algebraic Riccati equations (AREs) \cite[Thm. 14.7]{zhou1996robust}. Precisely, 
let $P,S$ be the unique positive semidefinite solutions to the following AREs
\begin{equation} \label{eq:ARE}
\begin{aligned}
    AP + PA^\tr -PC^\tr V^{-1}CP + W &= 0,\\
    A^\tr S + SA - SBR^{-1}B^\tr S + Q &= 0.
\end{aligned}
\end{equation}
Then, the parameters of an optimal controller to~\cref{eq:LQG} are 
\begin{align} \label{eq:optimal-ARE}
    A_{\mK}^\star = A - BM - LC,\, {B}_{\mK}^\star = L,\, {C}_{\mK}^\star = -M
\end{align}
where $L = PC^\tr V^{-1}$, $M = R^{-1}B^\tr S$. 
The optimal solution $(A_{\mK}^\star, {B}_{\mK}^\star, {C}_{\mK}^\star) $ is not unique in the state-space domain. Any similarity transformation leads to another equivalent optimal controller (they correspond to the same transfer function in the frequency domain).


\subsection{Problem Statement}

In this paper, we embrace the spirit of~\cite{fazel2018global,zheng2021analysis,mohammadi2019global,sun2021learning,malik2019derivative} and view the LQG problem~\cref{eq:LQG} from a modern optimization perspective. We consider the policy class $(A_\mK, B_\mK, C_\mK)$ in~\cref{eq:dynamical-controller}, and the closed-loop matrix becomes
\begin{equation}
\label{eq:closedloopmatrix}
A_{\mathrm{cl}}:=\begin{bmatrix}
    A & BC_{\mK} \\
    B_{\mK}C & A_{\mK}
\end{bmatrix} \in \mathbb{R}^{(n + q) \times (n + q)}.
\end{equation}
The set of internally stabilizing policies~\cite[Chapter 13]{zhou1996robust}~is   
\begin{equation*} 
    \mathcal{C}_{q} \!:=\! \left\{
    \left. \mK\!=\!\begin{bmatrix}
    0 & C_{\mK} \\
    B_{\mK} & A_{\mK}
    \end{bmatrix}
    \in \mathbb{R}^{(m+q) \times (p+q)} \right|\cref{eq:closedloopmatrix}~\text{is stable}\right\}\!.
\end{equation*}
Let $J_q(\mK):\mathcal{C}_q\rightarrow\mathbb{R}$ denote the corresponding LQG cost~\cref{eq:LQGcost} for each stabilizing policy in $\mathcal{C}_q$. It is known~\cite[Lemmas 2.3 \& 2.4]{zheng2021analysis} that this function $J_q(\mK)$ is real analytic on $\mathcal{C}_q$ and admits efficient computation.

\begin{lemma}\label{lemma:LQG_cost_formulation1}
Fix $q\in\mathbb{N}$ such that $\mathcal{C}_q\neq\varnothing$. Given $\mK\in\mathcal{C}_q$, we have
\begin{equation}\label{eq:LQG_cost_formulation1}
\begin{aligned}
J_q(\mK)
&=
\operatorname{tr}
\left(
Q_{\mathrm{cl},\mK} X_\mK\right) =
\operatorname{tr}
\left(
W_{\mathrm{cl},\mK} Y_\mK\right),
\end{aligned}
\end{equation}
where $X_{\mK}$ and $Y_{\mK}$ are the unique positive semidefinite 
solutions to the following Lyapunov equations
\begin{subequations}
\begin{align}
A_{\mathrm{cl}}X_{\mK} + X_{\mK}A_{\mathrm{cl}}^\tr +  W_{\mathrm{cl},\mK}
& = 0, \label{eq:LyapunovLQGX}
\\
A_{\mathrm{cl}}^\tr Y_{\mK} +  Y_{\mK}A_{\mathrm{cl}} +  Q_{\mathrm{cl},\mK}
& = 0, \label{eq:LyapunovLQGY}
\end{align}
\end{subequations}
where $A_{\mathrm{cl}}$ is defined in~\cref{eq:closedloopmatrix} and 
$$
Q_{\mathrm{cl},\mK} := \begin{bmatrix}
Q & 0 \\ 0 & C_{\mK}^\tr R C_{\mK}
\end{bmatrix}, \; W_{\mathrm{cl},\mK} := \begin{bmatrix}
W & 0 \\ 0 & B_{\mK} V B_{\mK}^\tr
\end{bmatrix}. 
$$
\end{lemma}

\Cref{lemma:LQG_cost_formulation1} works for stabilizing controllers of any order $q$. In this paper, we are mainly interested in characterizing the full-order case $\mathcal{C}_n$. Now, given the state dimension $n$, we can formulate the LQG problem~\cref{eq:LQG} into a constrained optimization problem
\begin{equation} \label{eq:LQG_reformulation}
    \begin{aligned}
        \min_{\mK} \quad &J_n(\mK) \\
        \text{subject to} \quad& \mK \in \mathcal{C}_n. 
    \end{aligned}
\end{equation}

An important notion of dynamical controllers is \textit{minimality}: a controller $(A_\mK, B_\mK, C_\mK)$ is minimal if $(A_\mK, B_\mK)$ is controllable and $(C_\mK, A_\mK)$ is observable.   
As revealed in~\cite{zheng2021analysis}, the optimization landscape of~\cref{eq:LQG_reformulation} is more complicated than that of LQR: 1) the feasible region $\mathcal{C}_n$ can have at most two disconnected components; 2) the cost function $J_n(\mK) $ is not coersive and not gradient dominant, and it can have suboptimal saddle points~\cite[Theorem 4.2]{zheng2021analysis}. Two nice features are 1) all sub-optimal saddle points correspond to \textit{non-minimal controllers} and 2) all stationary points that  correspond to \textit{minimal controllers} in $\mathcal{C}_n$ are globally optimal~\cite[Theorem 4.3]{zheng2021analysis}. 

Naive policy gradient methods can thus get stuck around sub-optimal saddle points. 
In this paper, we aim to provide further landscape characterizations of~\cref{eq:LQG_reformulation} and introduce a perturbed policy gradient method to escape bad stationary points of~\cref{eq:LQG_reformulation}. In particular, we first show that any stationary point of the LQG problem~\cref{eq:LQG_reformulation} after model reduction remain to be stationary, and then 
characterize the second-order behavior of $J_n(\mK)$ on a non-minimal stationary point. 
This motivates the design of our perturbed policy gradient method.



\section{Stationary Points and Their Hessians} \label{section:hessians}

The LQG problem~\cref{eq:LQG} has an inherent symmetry structure induced by the notion of similarity transformation. Let $\mathrm{GL}_q$ denote the set of $q \times q$ invertible matrices.  Given $q\geq 1$ such that $\mathcal{C}_q\neq\varnothing$,  the following map $\mathscr{T}_q:\mathrm{GL}_q\times\mathcal{C}_q\rightarrow\mathcal{C}_q$ represents similarity transformations 
\begin{equation}\label{eq:def_sim_transform}
\begin{aligned}
\mathscr{T}_q(T,\mK)
\coloneqq &\begin{bmatrix}
I_m & 0 \\
0 & T
\end{bmatrix}\begin{bmatrix}
0 & C_{\mK} \\
B_{\mK} & A_{\mK}
\end{bmatrix}\begin{bmatrix}
I_p & 0 \\
0 & T
\end{bmatrix}^{-1} \\
= &\begin{bmatrix}
0 & C_{\mK}T^{-1} \\
TB_{\mK} & TA_{\mK}T^{-1}
\end{bmatrix}.
\end{aligned}
\end{equation}
It is well-known that similarity transformations do not change the behavior of dynamical controllers and thus the LQG cost~\cref{eq:LQGcost} is invariant with respect to $\mathscr{T}_q(T,\mK)$, i.e., we have $J_q(\mK) = J_q\!\left(\mathscr{T}_q(T,\mK)\right), \forall \mK \in \mathcal{C}_q, T \in \mathrm{GL}_q$.

\subsection{Classification of stationary points}

The symmetry via similarity transformations brings rich and complicated landscape properties of~\cref{eq:LQG}. Here, we show that the underlying symmetry also allows a classification of stationary points of LQG~\cref{eq:LQG}. 
The lemma below gives an explicit relationship among the gradients of $J_q(\mK)$ at $\mK$ and $\mathscr{T}_q
\left(T,\mK
\right)$. 

\begin{lemma}[{\!\cite[Lemma 4.3]{zheng2021analysis}}] \label{lemma:gradient_simi_tran_linear}
Let $\mK=\begin{bmatrix}
0 & C_{\mK} \\ B_{\mK} & A_{\mK}
\end{bmatrix}\in \mathcal{C}_q$.
For any $T\in\mathrm{GL}_q$, we have
\begin{equation} \label{eq:Gradient_sim_transformation}
\left.\nabla J_q\right|_{\mathscr{T}_q \left(T,\mK \right)}
=\begin{bmatrix}
I_m & 0 \\
0 & T^{-\tr}
\end{bmatrix}
\cdot \left.\nabla J_q\right|_{\mK}
\cdot \begin{bmatrix}
I_p & 0 \\
0 & T^{\tr}
\end{bmatrix}.
\end{equation}
\end{lemma}

As expected, a direct consequence of \Cref{lemma:gradient_simi_tran_linear} is that a stationary point $\mK$ of $J_q$ remains to be stationary over $\mathcal{C}_q$ after any similarity transformation. We can further derive a classification of the stationary points of $J_n$ over the set of full-order controllers $\mathcal{C}_n$. 

\begin{theorem} \label{theorem:all-stationary-points}
Let  $\mK=\begin{bmatrix}
0 & C_{\mK} \\ B_{\mK} & A_{\mK}
\end{bmatrix}\in \mathcal{C}_n$ be  a stationary point of LQG~\cref{eq:LQG}, and let $\hat{\mK}=\begin{bmatrix}
0 & \hat{C}_{\mK} \\ \hat{B}_{\mK} & \hat{A}_{\mK}
\end{bmatrix}\in \mathcal{C}_q$ be  a minimal realization of $\mK$, where $q\leq n$ is the order of its minimal realization. Then, the following dynamical controller with any stable matrix $\Lambda \in \mathbb{R}^{(n-q) \times (n-q)}$ 
\begin{equation} \label{eq:stationary_nonglobally_K-augmented}
        \tilde{\mK}=\left[\begin{array}{c:cc}
    0 & \hat{C}_{\mK} &  0 \\[2pt]
    \hdashline
    \hat{B}_{\mK} & \hat{A}_{\mK} & 0 \\[-2pt]
    0 & 0 & \Lambda
    \end{array}\right] \in \mathcal{C}_{n}
\end{equation}
is a stationary point of~\cref{eq:LQG_reformulation}. If $q = n$ (i.e. $\mK$ itself is minimal), then $\mK$ is globally optimal. 
\end{theorem}

The fact of $\tilde{\mK}$~\cref{eq:stationary_nonglobally_K-augmented} being stationary of $J_n(\mK)$ seems to be expected, since $\tilde{\mK}$ and $\mK$ correspond to the same transfer function in the frequency domain and $\mK$ is in a higher dimensional space $\mathcal{C}_n$ than $\mathcal{C}_q$. The technical proof is not difficult, which combines the classical Kalman decomposition with \Cref{lemma:gradient_simi_tran_linear} and a result in~\cite[Theorem 4.1]{zheng2021analysis}. We provide the details in the appendix. The second part that $\mK$ is globally optimal if $q = n$ has been proved in~\cite[Theorem 4.3]{zheng2021analysis}.   

\Cref{theorem:all-stationary-points} shows all stationary points that correspond to non-minimal controllers admit a \emph{standard} parameterization as we defined in~\cref{eq:stationary_nonglobally_K-augmented}, which splits the controller state $\xi \in \mathbb{R}^n$ into 1) the controllable/observable (associated with $\hat{A}_{\mK},\hat{B}_{\mK},\hat{C}_{\mK}$ blocks) part and 2) non-controllable/non-observable (associated with $\Lambda$) part.
%
Furthermore, \Cref{theorem:all-stationary-points} indicates that all \textit{bad} stationary points of~\cref{eq:LQG} after model reduction are in the same form of~\cref{eq:stationary_nonglobally_K-augmented}. 
Thus, policy gradient methods only need to escape those bad saddle points of the form~\cref{eq:stationary_nonglobally_K-augmented}. This motivates our results in the next section.     

\begin{remark}[Non-minimal globally optimal controllers]
Note that a controller in the form of~\cref{eq:stationary_nonglobally_K-augmented} might still be globally optimal to~\cref{eq:LQG}; See \Cref{example:non-minimal} below. This happens when the solutions $(A_\mK^\star, B_\mK^\star, C^\star_\mK)$~\cref{eq:optimal-ARE} from the Riccati equations~\cref{eq:ARE} are not minimal, i.e. $(A_\mK^\star, B_\mK^\star)$ is not controllable or $(C_\mK^\star, A_\mK^\star)$ is not observable or both. We conjecture that a random LQG instance should have $(A_\mK^\star, B_\mK^\star, C^\star_\mK)$ in~\cref{eq:optimal-ARE} being minimal with probability one. An exact characterization is left for future work. 
\hfill $\square$
\end{remark}

\subsection{Hessian of stationary points}\label{ss: nonmin}

Once a policy gradient method reaches a stationary point, if the stationary point corresponds to a minimal controller, it has found a globally optimal solution to~\cref{eq:LQG}. If the stationary point does not correspond to a minimal controller, we can bring it into the form of~\eqref{eq:stationary_nonglobally_K-augmented}, for which we have the following characterization of its hessian.

\begin{theorem} \label{theorem:non_globally_optimal_stationary_point}
Consider a stationary point of $J_{n}(\mK)$ over $\mathcal{C}_{n}$ in the form of
\begin{equation} \label{eq:gradient_nonglobally_K}
        \tilde{\mK}
    =\left[\begin{array}{c:cc}
    0 & \hat{C}_{\mK} &  0 \\[2pt]
    \hdashline
    \hat{B}_{\mK} & \hat{A}_{\mK} & 0 \\[-2pt]
    0 & 0 & \Lambda
    \end{array}\right] \in \mathcal{C}_{n},
\end{equation}
with $\hat{A}_{\mK} \in \mathbb{R}^{q \times q}, \hat{B}_{\mK} \in \mathbb{R}^{q \times p}, \hat{C}_{\mK} \in \mathbb{R}^{m \times q}$, stable $\Lambda \in \mathbb{R}^{(n-q) \times (n-q)}$ and $q\leq n$.  
Let $X_{\operatorname{op}} \in \mathbb{S}^{n+q}_{+}$ and $Y_{\operatorname{op}}\in \mathbb{S}^{n+q}_{+}$ be the unique positive semidefinite solutions to the Lyapunov equations~\cref{eq:LyapunovLQGX} and~\cref{eq:LyapunovLQGY} with 
$
\hat{\mK}\!=\!\begin{bmatrix}
    0 & \hat{C}_{\mK} \\
    \hat{B}_{\mK} & \hat{A}_{\mK}
    \end{bmatrix} \in \mathcal{C}_q, 
$ respectively. 
Define a transfer function of size $p \times m$
\begin{equation} \label{eq:necessary-condition-G}
    \mathbf{G}(s) := C_{\mathrm{cl}}(s I - A_{_{\mathrm{cl}}}^\tr )^{-1} B_{\mathrm{cl}}.
\end{equation}
where $ A_{_{\mathrm{cl}}}$ is defined in \cref{eq:closedloopmatrix} with the $\hat{\mK}$ above, and  
$
C_{\mathrm{cl}}: = \bar{C}X_{\mathrm{op}}  +  V\bar{B}_{\mK}^\tr,  \;
B_{\mathrm{cl}}:=Y_{\mathrm{op}}\bar{B} +  \bar{C}_{\mK}^\tr R ,
$
with
\begin{align}
    \bar{C} &\!=\! \begin{bmatrix}
C \; \; 0
\end{bmatrix}\!\in\! \mathbb{R}^{p \times (n+q)}, \bar{B} \!=\! \begin{bmatrix}
B \\ 0
\end{bmatrix}  \!\in\! \mathbb{R}^{(n+q) \times m},  \label{eq:augmented-matrices} \\
\bar{C}_{\mK} &\!=\! \begin{bmatrix}
0 \;\; \hat{C}_{\mK}
\end{bmatrix} \!\in\! \mathbb{R}^{m \times (n+q)},\bar{B}_\mK \!=\! \begin{bmatrix}
0 \\ \hat{B}_{\mK}
\end{bmatrix}  \!\in\! \mathbb{R}^{(n+q) \times p}. \nonumber
\end{align}
The following statements hold.
\begin{enumerate}
    \item If $\tilde{\mK}$ in~\cref{eq:gradient_nonglobally_K} is globally optimal in $\mathcal{C}_{n}$, then the function $ \mathbf{G}(s)$ in~\cref{eq:necessary-condition-G} is identically zero $\forall s \in \mathbb{C}$.
    \item If $\mathbf{G}(s)$ in~\cref{eq:necessary-condition-G} is not a zero function, then $\tilde{\mK}$ is a strict saddle point (the Hessian of $J_n(\mK)$ at $\tilde{\mK}$ is indefinite) with probability one  when randomly choosing a stable and symmetric $\Lambda \in \mathbb{S}^{n-q}$.
    \item Let $\mathcal{Z}$ be the set of zeros of $\mathbf{G}(s)$, i.e.,$\mathcal{Z}
=\left\{
s\in\mathbb{C}\mid
\mathbf{G}(s) =0
\right\}.$ Given a stable and symmetric $\Lambda \in \mathbb{S}^{n-q}$, let $\operatorname{eig}(-\Lambda)$ denote the set of (distinct) eigenvalues of $-\Lambda$. If $\operatorname{eig}(-\Lambda) \nsubseteq \mathcal{Z}$, then the Hessian of $J_n(\mK)$ at $ \tilde{\mK}$ is indefinite. 
\end{enumerate}
\end{theorem}

\begin{proof}
Statements 1) and 2) are direct consequences of Statement 3). 
We give simple arguments below. 

$3) \Rightarrow 1)$: If $\tilde{\mK}$ in~\cref{eq:gradient_nonglobally_K} is globally optimal in $\mathcal{C}_{n}$, then the Hessian of  $J_n(\mK)$ at $\tilde{\mK}$ must be positive semidefinite. If $\mathbf{G}(s)$ is not identically zero, then its zero set $\mathcal{Z}$ is a set of finite points due to the fundamental theorem of algebra\footnote{Every non-zero, single-variable, degree $n$ polynomial with complex coefficients has, counted with multiplicity, exactly $n$ complex roots.}. Then, there exists a symmetric $\Lambda \in \mathbb{S}^{n-q}$ such that $\operatorname{eig}(-\Lambda) \nsubseteq \mathcal{Z}$, and thus its Hessian at $\tilde{\mK}$ is indefinite. This is contradicted with $\tilde{\mK}$ being globally optimal.

$3) \Rightarrow 2)$: If $ G(s)$ is not an identically zero function, then its zero set $\mathcal{Z}$ is a set of finite points. When choosing a stable and symmetric $\Lambda \in \mathbb{S}^{n-q}$ randomly, we have $\operatorname{eig}(-\Lambda) \nsubseteq \mathcal{Z}$ holds with probability one. Thus, $\tilde{\mK}$ is a strict saddle point with probability one.

The proof of Statement 3) exploits the bilinear property of the Hessian and the non-controllable/non-observable property to identify a two-by-two hessian block
\begin{equation*} 
    \begin{bmatrix}
    \operatorname{Hess}_{\,\tilde{\mK}}(\Delta^{(1)},\Delta^{(1)}) & \operatorname{Hess}_{\,\tilde{\mK}}(\Delta^{(1)},\Delta^{(2)})\\
    \operatorname{Hess}_{\,\tilde{\mK}}(\Delta^{(1)},\Delta^{(2)}) & \operatorname{Hess}_{\,\tilde{\mK}}(\Delta^{(2)},\Delta^{(2)})
    \end{bmatrix} \in \mathbb{S}^2
\end{equation*}
in which the diagonal entries are always zero. Using the Hessian calculation in~\cite[Lemma 4.3]{zheng2021analysis}, we then prove that if $\operatorname{eig}(-\Lambda) \nsubseteq \mathcal{Z}$, then the off-diagonal entries are non-zero. The Hessian of $J_n(\cdot)$ 
at $\tilde{\mK}$ is thus indefinite. Details are presented in the appendix. 
\end{proof}

Our \Cref{theorem:non_globally_optimal_stationary_point} includes the recent result~\cite[Theorem 4.2]{zheng2021analysis} as a special case in which the authors only consider a zero controller $\mK = 0$. Our main proof in the appendix, however, is motivated by that in~\cite[Theorem 4.2]{zheng2021analysis} with more complicated and careful calculations.

If the transfer function $\mathbf{G}(s)$ is not identically zero, then  $\tilde{\mK}$ in~\cref{eq:gradient_nonglobally_K} is a strict saddle point with probability one when randomly choosing $\Lambda$. Thus, we can apply the perturbed policy gradient method for ``escaping saddle'' \cite{jin2017escape}, so that the policy gradient iterations do not get stuck around these sub-optimal saddle points. We note that when $\mathbf{G}(s)$ is not identically zero, $\tilde{\mK}$ in~\cref{eq:gradient_nonglobally_K} may still have a zero Hessian (i.e., high-order saddle) if $\Lambda$ is chosen such that $\operatorname{eig}(-\Lambda) \subseteq \mathcal{Z}$; an explicit example is given \Cref{example:open-loop-stable} below. 
Therefore, our proposed perturbed policy gradient method for the LQG problem~\cref{eq:LQG} includes perturbations on $\Lambda$ as well as on the gradients. More details are given in \Cref{section:PGD}. 



\begin{remark}[Sufficiency of $\mathbf{G}(s) \equiv 0$ for global optimality and its interpretation] \label{remark:necessity}
\Cref{theorem:non_globally_optimal_stationary_point} holds with $q = n$, so $\mathbf{G}(s) \equiv 0, \forall s \in \mathbb{C}$ is also true when $\mK$ comes from the Riccati equations. In this case, we expect that $\mathbf{G}(s)$ in \cref{eq:necessary-condition-G} should have a nice control-theoretic interpretation. 
It is interesting to further investigate whether $\mathbf{G}(s) \equiv 0, \forall s\in \mathbb{C}$ is sufficient (or some other suitable conditions are needed) to certify the global optimality of $\tilde{\mK}$. 
\hfill $\square$
\end{remark}

We conclude this section by presenting three explicit LQG examples to illustrate \Cref{theorem:non_globally_optimal_stationary_point}.

\begin{example} \label{example:doyle}
We first consider the famous Doyle's LQG example~\cite{doyle1978guaranteed}, which has system matrices 
\begin{equation*}
    A = \begin{bmatrix}
        1 & 1\\
        0 & 1
    \end{bmatrix},\; B = \begin{bmatrix}
        0\\
        1
    \end{bmatrix}, \; C = \begin{bmatrix} 1 & 0 \end{bmatrix},\;
\end{equation*}
and performance weights 
\begin{equation*}
  W = 5\begin{bmatrix}
    1 & 1 \\
    1 & 1
    \end{bmatrix}, \; V = 1, \; Q = 5\begin{bmatrix}
    1 & 1 \\
    1 & 1
    \end{bmatrix}, \; R = 1. 
\end{equation*}
The globally optimal LQG controller from \cref{eq:optimal-ARE} is 
\begin{equation*} 
    A_{\mK}^\star = \begin{bmatrix}
       -4 &    1 \\
  -10 &-4
    \end{bmatrix}, \; {B}_{\mK}^\star = \begin{bmatrix}
    5\\5
    \end{bmatrix},  {C}_{\mK}^\star = \begin{bmatrix}
    -5 & -5
    \end{bmatrix}.
\end{equation*}
The Hessian $J_2(\mK)$ at ${\mK}^\star = \left[\begin{array}{c:c}
0 & {C}_{\mK}^\star \\[2pt]
    \hdashline {B}_{\mK}^\star & A_{\mK}^\star
\end{array}\right]
\in \mathcal{C}_{2}$ is
positive semidefinite and has eigenvalues 
$
\lambda_1 =     8.1111 \times 10^5,  \lambda_2 =   6\,133.9, \lambda_3 =    131.2,  \lambda_4 =     6.36,  \lambda_5 = \cdots = \lambda_8 = 0
$
(see Appendix C for details). Four zero eigenvalues are expected due to the symmetry induced by the similarity transformation~\cite[Lemma 4.6]{zheng2021analysis}.   
We further compute the matrices in~\cref{eq:augmented-matrices} (their values can be found in the appendix), and we have 
    $$
    \begin{aligned}
    &(\bar{C}X_{\mathrm{op}}  +  V\bar{B}_{\mK}^\tr)(s I - A_{_{\mathrm{cl}}}^\tr )^{-1}Y_{\mathrm{op}} \bar{B} \\
    = &\frac{-12.5 s^3 - 604.2 s^2 - 1712 s - 566.7}{    s^4 + 6 s^3 + 11 s^2 + 6 s + 1}, 
      \end{aligned}
    $$
    and
    $$
      \begin{aligned}
        &(\bar{C}X_{\mathrm{op}}  +  V\bar{B}_{\mK}^\tr)(s I - A_{_{\mathrm{cl}}}^\tr )^{-1} \bar{C}_{\mK}^\tr R\\
        =& \frac{  12.5 s^3 + 604.2 s^2 + 1713 s + 566.7}{  s^4 + 6 s^3 + 11 s^2 + 6 s + 1}. 
    \end{aligned}
    $$
    Thus, we have 
    $$
    \mathbf{G}(s) \!=\! (\bar{C}X_{\mathrm{op}}  +  V\bar{B}_{\mK}^\tr)(s I - A_{_{\mathrm{cl}}}^\tr )^{-1} (Y_{\mathrm{op}}\bar{B} +  \bar{C}_{\mK}^\tr R ) \equiv 0.
    $$
This result that $\mathbf{G}(s)$ being identically zero is expected from~\Cref{theorem:non_globally_optimal_stationary_point} since $\mK^\star$ is globally optimal.  
\hfill $\square$
\end{example}

We then consider~\cite[Example 7]{zheng2021analysis} for which the globally optimal LQG controller is non-minimal in $\mathcal{C}_n$.

\begin{example} \label{example:non-minimal}
Consider an LQG instance with  matrices
$$
A=\begin{bmatrix}
0 & -1 \\ 1 & 0
\end{bmatrix},
\;
B = \begin{bmatrix}
1 \\ 0
\end{bmatrix},
\;
C = \begin{bmatrix}
1 & -1
\end{bmatrix}
$$
and performance weights
$$
W = \begin{bmatrix}
1 & -1 \\ -1 & 16
\end{bmatrix},
\;
V = 1, \; Q=\begin{bmatrix}
4 & 0 \\ 0 & 0
\end{bmatrix},
\;
R = 1.
$$
The globally optimal controller from \cref{eq:optimal-ARE} is given by
\begin{equation*} 
A_{\mK}^\star = \begin{bmatrix}
-3 & 0 \\
5 & -4
\end{bmatrix},
{B}_{\mK}^\star
=
\begin{bmatrix}
1 \\ -4
\end{bmatrix},
{C}_{\mK}^\star =
\begin{bmatrix}
-2 & 0
\end{bmatrix}.
\end{equation*}
It is easy to verify that $(C_{\mK}^*,A_{\mK}^*)$ is not observable. The Hessian of $J_2(\mK)$ at  ${\mK}^\star \in \mathcal{C}_{2}$ is positive semidefinite with eigenvalues as
$
\lambda_1 =  581.5529, \lambda_2 =  7.1879, \lambda_3 =  0.2592, \lambda_4 = \cdots = \lambda_8 = 0.
$ (See Appendix D for details).  Four zero eigenvalues are expected, due to the symmetry by similarity transformations, and the other zero is caused by the unobservablility of $(C_{\mK}^*,A_{\mK}^*)$. Consider two reduced-order controllers 
$$
    \mK_1 = \left[\begin{array}{c:c}
    0 & -2 \\[2pt]
    \hdashline
    1 &  -3
    \end{array}\right] \in \mathcal{C}_1, \quad  \mK_2 = \left[\begin{array}{c:c}
    0 & 0.5 \\[2pt]
    \hdashline
    -4 &  -3
    \end{array}\right] \in \mathcal{C}_1,
$$
both of which are globally optimal. Thus, the following two full-order controllers
$$
\tilde{\mK}_1
    =\left[\begin{array}{c:cc}
    0 & -2 &  0 \\[2pt]
    \hdashline
    1 & -3 & 0 \\[-2pt]
    0 & 0 & \Lambda
    \end{array}\right], \quad \tilde{\mK}_2
    =\left[\begin{array}{c:cc}
    0 & 0.5 &  0 \\[2pt]
    \hdashline
    -4 & -3 & 0 \\[-2pt]
    0 & 0 & \Lambda
    \end{array}\right],  
$$
are globally optimal as well. From \Cref{theorem:non_globally_optimal_stationary_point}, we expect $\mathbf{G}(s) \equiv 0$ for both $\tilde{\mK}_1$ and $\tilde{\mK}_2$. 
For both of them, we can compute (details are in Appendix D) that 
    $$
    \begin{aligned}
    (\bar{C}X_{\mathrm{op}}  +  V\bar{B}_{\mK}^\tr)(s I - A_{_{\mathrm{cl}}}^\tr )^{-1}Y_{\mathrm{op}} \bar{B}&= \frac{26.5s +   56.5}{(s + 1)^2} \\
        (\bar{C}X_{\mathrm{op}}  +  V\bar{B}_{\mK}^\tr)(s I - A_{_{\mathrm{cl}}}^\tr )^{-1} \bar{C}_{\mK}^\tr R&= -\frac{26.5s +   56.5}{(s + 1)^2}. 
    \end{aligned}
    $$
    Thus, we have the expected result from \Cref{theorem:non_globally_optimal_stationary_point}
    that $
    \mathbf{G}(s)=(\bar{C}X_{\mathrm{op}}  +  V\bar{B}_{\mK}^\tr)(s I - A_{_{\mathrm{cl}}}^\tr )^{-1} (Y_{\mathrm{op}}\bar{B} +  \bar{C}_{\mK}^\tr R ) \equiv 0. 
    $
     \hfill\qed
\end{example}


Finally, we consider an LQG problem with a high-order saddle point. This high-order saddle point is predicted in \Cref{theorem:non_globally_optimal_stationary_point} and \cite[Theorem 4.2]{zheng2021analysis}.

\begin{example} \label{example:open-loop-stable}
Consider an LQG instance with an open-loop stable system, in which the problem data are
$$
\begin{aligned}
    A &= \begin{bmatrix}
   -0.5 & 0 \\ 0.5 & -1
\end{bmatrix},\ B = \begin{bmatrix}
   -1 \\ 1
\end{bmatrix},\ C = \begin{bmatrix}
   -\frac{1}{6} & \frac{11}{12}
\end{bmatrix}, 
\end{aligned}
$$
with weight matrices $W = Q = \Ib_2,\ V=R=1.$
Since this example is open-loop stable, \cite[Theorem 4.2]{zheng2021analysis} guarantees that $\tilde{\mK} = \left[\begin{array}{c:c}
0 & 0 \\[2pt]
    \hdashline 0 & \Lambda
\end{array}\right]
\in \mathcal{C}_{2}$ with any stable $ \Lambda \in \mathbb{R}^{2 \times 2}$ is a stationary point. At this controller, we can compute that the transfer function in~\cref{eq:necessary-condition-G} is 
$$
\mathbf{G}(s) = \frac{5(2s-1)}{108(2s^2 + 3s +1)}.
$$
The zero set $\mathcal{Z} = \{0.5\}$ contains a single value.~For any stable $\Lambda$ with $\operatorname{eig}(-\Lambda) \nsubseteq \mathcal{Z}$, the Hessian is indefinite by \Cref{theorem:non_globally_optimal_stationary_point}. For instance,  with $\Lambda = - \mathrm{diag}(0.5,0.1)$, the Hessian is indefinite with eigenvalues $\lambda_1 = 0.0561,\lambda_2 = -0.0561, \lambda_i = 0, i = 3, \ldots, 8$.
 However, we can check that if $\Lambda = -0.5\Ib_2$, (i.e. $
     A_{\mK} = -0.5 \Ib_2,\ B_{\mK} = 0,\ C_{\mK} = 0
$), its Hessian is degenerated to zero, implying that it is a high-order saddle. 
Our proposed perturbed gradient descent algorithm in the next section can escape this type of high-order saddles efficiently. 
\hfill\qed
\end{example}

\section{Perturbed policy gradient method} \label{section:PGD}
Inspired by \Cref{theorem:all-stationary-points,theorem:non_globally_optimal_stationary_point}, we introduce~a~novel perturbed policy gradient method that combines~a~structural perturbation on $\Lambda$ in \cref{eq:gradient_nonglobally_K} with a standard perturbation on gradients~\cite{ge2015escaping,jin2017escape}. Numerical results confirm that our perturbed policy gradient method can escape high-order saddles more efficiently, than either vanilla policy gradient or standard perturbed policy gradient~\cite{jin2017escape}.

\subsection{Algorithm setup}

Recent work has established that variations of gradient descent can escape strict saddle-points -- points~at~which the minimum eigenvalue of the Hessian is strictly negative. For example, stochastic gradient descent \cite{ge2015escaping}, gradient descent with appropriate random perturbation~\cite{jin2017escape} or with cubic regularization sub-oracle~\cite{carmon2016gradient} are proven to escape strict saddles and visit an approximate local minimum in polynomial time with high probability.

Our method combines the standard perturbed gradient descent \cite[Algorithm 2]{jin2017escape} with an additional oracle of random structural perturbation on $\Lambda$. 
Our perturbed policy gradient descent is listed in \Cref{algo:1}. We note that \Cref{algo:1} is a \emph{prototype} algorithm in the sense that some quantities (e.g., order-reduction, gradient and Hessian~Lipschitz constants) of the LQG problem require more investigations. Convergence conditions and further quantitative analysis of our algorithm are also left for future work. 

The high-level ideas are described below. 
%

\vspace{-1mm}

\begin{itemize}
\setlength{\itemsep}{3pt}
    \item When the gradient of a controller $\K_t$ is close~to~zero, we check whether it is minimal, i.e., the smallest Hankel singular value of the controllability/observability matrix is bounded away from zero (for the connection with Hankel singular values and controllability/observability, please refer \cite[Chapter 7]{zhou1996robust}).
    \item If $\K_t$ is controllable and observable (i.e., minimal), it is close to be  globally optimal by \Cref{theorem:all-stationary-points}. We terminate the algorithm. 
    \item If $\K_t$ is non-minimal, we perform a minimal realization (e.g., Kalman decomposition or balance realization) to get a controller in the form of \cref{eq:stationary_nonglobally_K-augmented}.
    \item We then choose a symmetric and stable $\Lambda$ randomly. From \cref{theorem:non_globally_optimal_stationary_point}, we expect that the resulting controller is close to a strict saddle point. 
    \item We apply a random perturbation on the gradients. The random perturbation is i.i.d. Gaussian variables, with small magnitudes, added to each entry of $A_{\mK},B_{\mK},C_{\mK}$ such that the controller is still stabilizing.  We run a few gradient descent iterations afterwards. We expect that these gradient descent iterations will escape from the strict saddle point.  
    

\item We terminate the algorithm when the algorithm reaches the predefined number of steps $T$. 
\end{itemize}

\vspace{-1mm}

 \Cref{algo:1} can escape a large class of (but not all) high-order saddles at which $\mathbf{G}(s)$ in~\cref{eq:necessary-condition-G} is not identically zero.  
When \Cref{algo:1} terminates, it is likely to produce an approximately global minimum or return a point at which the transfer function $\mathbf{G}(s)$ in~\cref{eq:necessary-condition-G} is close to zero. In the later case, the point may not be globally optimal, and this is related to the sufficiency of $\mathbf{G}(s) \equiv 0$ for  global optimality in \Cref{remark:necessity}.  

\begin{algorithm}[tp!]
\caption{Perturbed policy gradient}\label{algo:1}
\begin{algorithmic}[1]
\REQUIRE{1) Loss $J(\K)$ with its gradient. 2) Thresholds $\gth$, $\iota$. 3) Constant  $T$, $\tau$, step size $\eta$. 4) Function $\lbmin(\K)$ that returns the minimum singular value of the Hankel matrix of $\K$. 5) Function $\redord(\K)$ that finds the approximate order of $\K$.} 
\STATE{Set $t=0$, $t_{\text{perturb}} = -\tau- 1$ and initialize a stabilizing controller $\K_0$.}
\WHILE {$t\leq T$}    
    \IF {$\|\nabla J(\K_t)\| \le \gth$ and $\lbminK \ge \iota$}
     \STATE{\textbf{output} $\K$}
    \ELSIF {$\|\nabla J(\K_t)\| \le \gth$ and $\lbminK \le \iota$ and $t-t_{\text{perturb}}>\tau$}
     \STATE{$\hat{\K}_t, q_t \leftarrow \redord(\K_t)$ where $q_t$ is the order after model reduction; \label{algline:perturbe_Lambda-1}} 
     \STATE{$\Lambda_t\leftarrow \lambda I_{n-q_t}$ with $\lambda <0$ randomly selected; \label{algline:perturbe_Lambda}} 
     \STATE {$\K_t \leftarrow \diag(\hat{\K}_t, \Lambda_t)$ as in~\cref{eq:gradient_nonglobally_K} (Theorem \ref{theorem:non_globally_optimal_stationary_point}); \label{algline:perturbe_Lambda-3}} 
     \STATE{$\K_t\leftarrow \K_t + \xi_t$ with $\xi_t$ uniformly sampled from $\mathbb{B}_{\K_t}(r)$; \label{algline:perturbe_epsilon}}
     \STATE{$t_\text{perturb}\leftarrow t$;}
    \ENDIF
    \STATE {$\K_{t+1} \leftarrow \K_t - \eta \nabla J(\K_{t});$}
    \STATE{$t\leftarrow t+1$;} 
   \ENDWHILE    
\end{algorithmic}
\end{algorithm}

\subsection{Numerical results}

We implement \Cref{algo:1}, and consider \Cref{example:open-loop-stable} for numerical comparison with three other algorithms:
\begin{enumerate}
    \item Vanilla policy gradient; 
    \item Standard perturbed policy gradient \cite{jin2017escape} (with no perturbation on dynamics $\Lambda$, i.e., no Lines~\ref{algline:perturbe_Lambda-1}-\ref{algline:perturbe_Lambda-3} in \Cref{algo:1});
    \item Perturb the dynamics $\Lambda$ but with no perturbation on gradients (i.e., no Line~\ref{algline:perturbe_epsilon} in \Cref{algo:1}.). 
\end{enumerate}

The globally optimal controller from \cref{eq:optimal-ARE} for the LQG instance in \Cref{example:open-loop-stable}  is 
    \begingroup
    \setlength\arraycolsep{0.8pt}
\def\arraystretch{0.9}
\begin{align*}
    A_{\mK}^\star \!=\! \begin{bmatrix}
       -1.10 &   0.13\\
    1.19 &  -1.64\\
    \end{bmatrix},\ {B}_{\mK}^\star \!=\! \begin{bmatrix}
       0.11 \\
    0.45
    \end{bmatrix}, 
    {C}_{\mK}^\star \!=\! \begin{bmatrix}
       0.62  & -0.22
    \end{bmatrix}.
\end{align*}
\endgroup
To illustrate the performance of different algorithms, we initialize the controller at 
\begin{align} \label{eq:initialK0}
    A_{\mK,0} \!=\! \text{-}0.5 \Ib_2,
    B_{\mK,0} \!=\!  \begin{bmatrix}
   0\\0.01
\end{bmatrix}, C_{\mK,0} \!=\! \begin{bmatrix}
   0, \text{-}0.01 
\end{bmatrix}.
\end{align}
As discussed in \Cref{example:open-loop-stable}, this initial point is close to a high-order saddle $
    A_{\mK} = -0.5 \Ib_2, B_{\mK} = 0, C_{\mK} = 0.
$
 We add a perturbation to the first iteration and run gradient descent with the fixed step size. The perturbations are different, as discussed at the beginning of this section. 
 
 The results are shown in \Cref{fig: exp}: the left sub-figure shows the suboptimality gap, and the right one shows the norm of graidents at each iteration. 
 Our \Cref{algo:1} implements both perturbations: 1) identifying an one-dimensional $\Lambda$ as in the standard form \eqref{eq:gradient_nonglobally_K} and change it randomly, and 2) randomly perturb all variables with a small quantity $0.01$. As shown in \Cref{fig: exp}, our \Cref{algo:1} can escape this high-order saddle faster than the  other three algorithms, including  standard PGD in \cite{jin2017escape} (in which no perturbation on  $\Lambda$~was~applied). 
 
\setlength{\textfloatsep}{.5em}
\begin{figure}[t]
\centering
\includegraphics[width=0.5\textwidth]{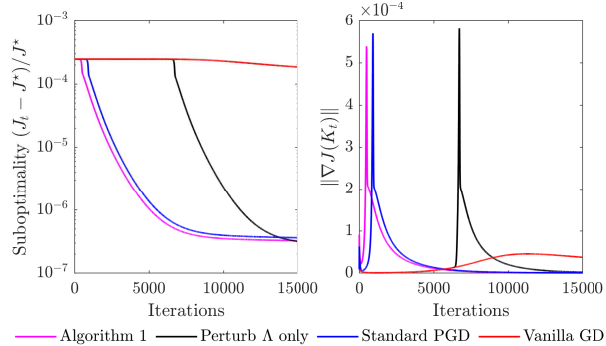}
\caption{Comparison of different perturbed and Vanilla policy gradient (PG) methods: Our \Cref{algo:1}, Vanilla GD, standard PGD in \cite{jin2017escape} (with no perturbation on dynamics $\Lambda$), and PGD with perturbation on dynamics $\Lambda$ only. These algorithms all start from the same point \cref{eq:initialK0} near a high-order saddle, and applied fixed step-size gradient descent iterations. Left: suboptimality $\frac{J(\mK_t) - J^\star}{J^\star}$; Right: norm of gradients $\|\nabla J(\mK_t)\|$.
}\label{fig: exp}
\end{figure}

\balance
\section{Conclusions} \label{section:conclusion}

We have proposed a novel PGD algorithm (cf. \Cref{algo:1}) to escape high-order saddles of LQG. 
Our PGD algorithm combines the inherent structure of LQG control with standard perturbation on gradients. We have shown the structure of all stationary points after model reduction (cf. \Cref{theorem:all-stationary-points}). We have also introduced a reparameterization procedure with an intriguing transfer function $\mathbf{G}(s)$ at any stationary point (cf. \Cref{theorem:non_globally_optimal_stationary_point}). If $\mathbf{G}(s)\not\equiv 0$, we can certify that the high-order saddle can be made as a strict saddle by the reparameterization. 
Numerical simulations confirmed that \Cref{algo:1} combining the reparameterization with random perturbation on gradients can accelerate the speed of escaping high-order saddles. 
Ongoing and future directions include quantitative analysis of \Cref{algo:1}. {We are also interested in the sufficiency of $\mathbf{G}(s) \equiv 0$ (or~other~conditions are needed) for global optimality of LQG (see \Cref{remark:necessity}).}


\bibliographystyle{IEEEtran}
\bibliography{references}

\onecolumn 
\appendices

\section*{Appendix}

In this appendix, we complete the proofs to \Cref{theorem:all-stationary-points,theorem:non_globally_optimal_stationary_point}, and present some further details on the Hessian computations in \Cref{example:doyle,example:non-minimal,example:open-loop-stable}. 

\subsection{Proof of \Cref{theorem:all-stationary-points}}

Our proof for \Cref{theorem:all-stationary-points} is built on the classical Kalman decomposition for linear time-invariant (LTI) systems and a recent result in~\cite[Theorem 4.1]{zheng2021analysis}. We first recall the celebrated Kalman canonical decomposition.

\begin{theorem}[Kalman decomposition{~\cite[Theorem 3.10]{zhou1996robust}}] \label{theorem:kalman}
Given any LTI system of the form
\begin{equation}
    \begin{aligned}
      \dot{x} &= Ax  + Bu \\
      y &=Cx + Du,
    \end{aligned}
\end{equation}
there exists a nonsingular coordinate transformation $\tilde{x} = Tx$ such that 
\begin{equation} \label{eq:kalman-decomposition}
    \begin{aligned}
      \begin{bmatrix}
         \dot{\tilde{x}}_{\mathrm{co}} \\ \dot{\tilde{x}}_{\mathrm{c\bar{o}}} \\\dot{\tilde{x}}_{\mathrm{\bar{c}o}} \\ \dot{\tilde{x}}_{\mathrm{\bar{c}\bar{o}}}
      \end{bmatrix} &= 
      \begin{bmatrix}
         \tilde{A}_{\mathrm{co}} & 0 & \tilde{A}_{13} & 0 \\
         \tilde{A}_{21} & \tilde{A}_{\mathrm{c\bar{o}}} & \tilde{A}_{23} & \tilde{A}_{24} \\
         0 & 0 & \tilde{A}_{\mathrm{\bar{c}o}} & 0 \\
         0 & 0 & \tilde{A}_{43} & \tilde{A}_{\mathrm{\bar{c}\bar{o}}}
               \end{bmatrix}
         \begin{bmatrix}
         {\tilde{x}}_{\mathrm{co}} \\ {\tilde{x}}_{\mathrm{c\bar{o}}} \\{\tilde{x}}_{\mathrm{\bar{c}o}} \\ {\tilde{x}}_{\mathrm{\bar{c}\bar{o}}}
      \end{bmatrix} + \begin{bmatrix}
         \tilde{B}_{\mathrm{co}} \\ \tilde{B}_{\mathrm{c\bar{o}}} \\ 0 \\ 0
      \end{bmatrix} u \\
      y &= \begin{bmatrix}
         \tilde{C}_{\mathrm{co}} & 0 &\tilde{C}_{\mathrm{\bar{c}{o}}} & 0 
      \end{bmatrix}\begin{bmatrix}
         {\tilde{x}}_{\mathrm{co}} \\ {\tilde{x}}_{\mathrm{c\bar{o}}} \\{\tilde{x}}_{\mathrm{\bar{c}o}} \\ {\tilde{x}}_{\mathrm{\bar{c}\bar{o}}}
      \end{bmatrix} + D u, 
    \end{aligned}
\end{equation}
where the vector $\tilde{x}_{\mathrm{co}}$ is controllable and observable, $\tilde{x}_{\mathrm{c\bar{o}}}$ is controllable but unobservable, $\tilde{x}_{\mathrm{\bar{c}{o}}}$ is observable but uncontrollable, and $\tilde{x}_{\mathrm{\bar{c}\bar{o}}}$ is uncontrollable and unobservable.
Moreover, the transfer matrix from $u$ to $y$ is given by
$$
    \mathbf{G}(s) = C (sI - A)^{-1}B + D = \tilde{C}_{\mathrm{co}} (sI - \tilde{A}_{\mathrm{co}})^{-1}\tilde{B}_{\mathrm{co}} + D. 
$$
\end{theorem}

The dimension of $\tilde{x}_{\mathrm{co}}$ is the same as the dimension of the minimal realization of $ \mathbf{G}(s)$. The result in~\cite[Theorem 4.1]{zheng2021analysis} states that any stationary point of $J_{q}$ can be augmented to be a stationary point of $J_{q+q'}$ for any $q'>0$ over $\mathcal{C}_{q+q'}$ with the same LQG cost.

\begin{theorem}[{\!\cite[Theorem 4.1]{zheng2021analysis} }]  \label{theorem:high-order-stationary-point}
Let $q\geq 1$ be arbitrary. Suppose there exists $\hat{\mK}=\begin{bmatrix}
0 & \hat{C}_{\mK} \\ \hat{B}_{\mK} & \hat{A}_{\mK}
\end{bmatrix}
\in \mathcal{C}_{q}$ such that $\nabla J_{q}(\hat{\mK})=0$. 
Then for any $q'\geq 1$ and any stable $\Lambda\in\mathbb{R}^{q'\times q'}$, the following controller
\begin{equation*} 
         \tilde{\mK}
    =\left[\begin{array}{c:cc}
    0 & \hat{C}_{\mK} &  0 \\[2pt]
    \hdashline
    \hat{B}_{\mK} & \hat{A}_{\mK} & 0 \\[-2pt]
    0 & 0 & \Lambda
    \end{array}\right] \in \mathcal{C}_{q+q'}
\end{equation*}
is a stationary point of $J_{q+q'}$ over $\mathcal{C}_{q+q'}$ satisfying $J_{q+q'}\big( \tilde{\mK}\big)
=J_{q}(\hat{\mK})$.
\end{theorem}

We are now ready to prove \Cref{theorem:all-stationary-points}.

\textbf{Proof of \Cref{theorem:all-stationary-points}: } Let  $\mK=\begin{bmatrix}
0 & C_{\mK} \\ B_{\mK} & A_{\mK}
\end{bmatrix}\in \mathcal{C}_n$ be a stationary point of LQG~\cref{eq:LQG}. According to \Cref{theorem:kalman}, we can perform a Kalman decomposition using a similarity transformation $T$, i.e.,
$$
\begin{bmatrix}
0 & C_{\mK} \\ B_{\mK} & A_{\mK}
\end{bmatrix} \to \begin{bmatrix}
0 & C_{\mK}T^{-1} \\ TB_{\mK} & TA_{\mK}T^{-1}
\end{bmatrix} = \begin{bmatrix}
0 & \tilde{C}_{\mK} \\ \tilde{B}_{\mK} & \tilde{A}_{\mK}
\end{bmatrix} = \tilde{\mK}
$$
where the system $(\tilde{A}_{\mK},\tilde{B}_{\mK}, \tilde{C}_{\mK},0)$ is in the standard Kalman decomposition of the form~\eqref{eq:kalman-decomposition}. By \Cref{lemma:gradient_simi_tran_linear}, we further have 
$$
\left.\nabla J_n\right|_{\tilde{\mK}} = \begin{bmatrix}
I_m & 0 \\
0 & T^{-\tr}
\end{bmatrix}
\cdot \left.\nabla J_q\right|_{\mK}
\cdot \begin{bmatrix}
I_p & 0 \\
0 & T^{\tr}
\end{bmatrix} = 0.
$$
Therefore, the new point $\hat{\mK}$ in the Kalman decomposition remains to be a stationary point of $J_n$ in $\mathcal{C}_n$. This is expected since $\tilde{\mK}$ and ${\mK}$ corresponds to the same transfer function. 

Denote the controllable and observable components in $\tilde{\mK} = \begin{bmatrix}
0 & \tilde{C}_{\mK} \\ \tilde{B}_{\mK} & \tilde{A}_{\mK}
\end{bmatrix} \in \mathcal{C}_n$ as  $\tilde{\mK}_{\mathrm{co}} =  \begin{bmatrix}
0 & \tilde{C}_{\mK,\mathrm{co}} \\ \tilde{B}_{\mK,\mathrm{co}} & \tilde{A}_{\mK,\mathrm{co}}
\end{bmatrix} \in \mathcal{C}_q$, where $q$ is the dimension of controllable and observable states (i.e., the dimension of its minimal realization). Since $\tilde{\mK}$ is a stationary point of $J_n(\mK)$ over $\mathcal{C}_n$, the directional derivatives of $J_n(\mK)$ at $\tilde{\mK}_{\mathrm{co}}$ are all zero. Therefore, it is not difficult to see that $\tilde{\mK}_{\mathrm{co}}$ is a stationary point of $J_q(\mK)$ over  $\mathcal{C}_q$. 

Now, given an arbitrary minimal realization of $\mK$ as $\hat{\mK}=\begin{bmatrix}
0 & \hat{C}_{\mK} \\ \hat{B}_{\mK} & \hat{A}_{\mK}
\end{bmatrix}\in \mathcal{C}_q$, there exists a unique similarity transformation $T_q \in \mathbb{R}^{q \times q}$~\cite[Theorem 3.17]{zhou1996robust}  such that
$$
    \begin{bmatrix}
0 & \hat{C}_{\mK} \\ \hat{B}_{\mK} & \hat{A}_{\mK}
\end{bmatrix} = \begin{bmatrix}
0 & \tilde{C}_{\mK,\mathrm{co}}T_q^{-1} \\ T_q\tilde{B}_{\mK,\mathrm{co}}& T_q\tilde{A}_{\mK,\mathrm{co}}T_q^{-1}
\end{bmatrix}.  
$$
 \Cref{lemma:gradient_simi_tran_linear} ensures that this minimal realization $\hat{\mK} \in \mathcal{C}_q$ is a stationary point of $J_q(\mK)$ since $\tilde{\mK}_{\mathrm{co}}$ is stationary. Applying~\Cref{theorem:high-order-stationary-point} with $q' = n - q$ leads to the desired result in~\cref{eq:stationary_nonglobally_K-augmented}. 
 
 The statement that $\mK$ is globally optimal when $q = n$ has been shown in~\cite[Theorem 4.3]{zheng2021analysis} where the authors have shown that all minimal full-order stationary points are in the form~\cref{eq:optimal-ARE} up to a similarity transformation. This completes the proof.  \hfill $\square$


\subsection{Proof of \Cref{theorem:non_globally_optimal_stationary_point}}

Before presenting the proof, we recall a few definitions. We define a linear space 
\begin{equation}\label{eq:def_Vq}
\mathcal{V}_q\coloneqq \left\{
\left.\begin{bmatrix}
    D_{\mK} & C_{\mK} \\
    B_{\mK} & A_{\mK}
    \end{bmatrix}
    \in \mathbb{R}^{(m+q) \times (p+q)} \right|\;  D_{\mK} = 0_{m\times p}\right\}.
\end{equation}
Let $\mK$ be any controller in $\mathcal{C}_n \subset \mathcal{V}_n$, and we use $\operatorname{Hess}_{\,\mK}:\mathcal{V}_n\times\mathcal{V}_n\rightarrow\mathbb{R}$ to denote the bilinear form of the Hessian of $J_n$ at $\mK$, so that for any $\Delta\in\mathcal{V}_n$, we have 
\begin{equation}
J_n(\mK+\Delta)
=J_n(\mK)+
\operatorname{tr}
\left(
\nabla J_q(\mK)^\tr \Delta\right)
+
\frac{1}{2}
\operatorname{Hess}_{\,\mK}(\Delta,\Delta)
+o(\|\Delta\|_F^2)
\end{equation}
as $\|\Delta\|_F\rightarrow 0$. Obviously, $\operatorname{Hess}_{\,\mK}$ is symmetric in the sense that $\operatorname{Hess}_{\,\mK}(x,y)=\operatorname{Hess}_{\,\mK}(y,x)$ for all $x,y\in\mathcal{V}_n$. 

We recall~\cite[Lemma 4.3]{zheng2021analysis} that presents the hessian calculations of the LQG problem~\cref{eq:LQG} by solving three Lyapunov equations.
\begin{lemma}[{\!\!\cite[Lemma 4.3]{zheng2021analysis}}]\label{lemma:Jn_Hessian}
Fix $q\geq 1$ such that $\mathcal{C}_q\neq\varnothing$.
Let $\mK=\begin{bmatrix}
0 & C_{\mK} \\
B_{\mK} & A_{\mK}
\end{bmatrix} \in \mathcal{C}_q$. Then for any $\Delta=\begin{bmatrix}
0 & \Delta_{C_\mK} \\
\Delta_{B_\mK} & \Delta_{A_\mK}
\end{bmatrix}\in\mathcal{V}_q$, we have
$$
\begin{aligned}
\operatorname{Hess}_{\,\mK}(\Delta,\Delta)
=\ &
2\operatorname{tr}
\Bigg(
2
\begin{bmatrix}
0 & B\Delta_{C_\mK} \\
\Delta_{B_\mK} C & \Delta_{A_\mK}
\end{bmatrix}
X'_{\mK,\Delta}\cdot Y_{\mK}
+2\begin{bmatrix}
0 & 0 \\ 0 & {C_{\mK} }^\tr R \Delta_{C_\mK}
\end{bmatrix}\cdot X'_{\mK,\Delta}
\\
& \qquad\qquad
+\begin{bmatrix}
0 & 0 \\
0 & \Delta_{B_\mK}V\Delta_{B_\mK}^\tr
\end{bmatrix} Y_{\mK}
+
\begin{bmatrix}
0 & 0 \\ 0 & \Delta_{C_\mK}^\tr R \Delta_{C_\mK}
\end{bmatrix}X_{\mK}
\Bigg),
\end{aligned}
$$
where $X_{\mK}$ and $Y_{\mK}$ are the solutions to the Lyapunov equations~\eqref{eq:LyapunovLQGX} and~\eqref{eq:LyapunovLQGY}, and $X'_{\mK,\Delta}\in\mathbb{R}^{(n+q)\times(n+q)}$ is the solution to the following Lyapunov equation
\begin{equation} \label{eq:Lyapunov_hessian}
    \begin{bmatrix}
A & BC_{\mK} \\
B_{\mK} C & A_{\mK} 
\end{bmatrix} X'_{\mK,\Delta}
+X'_{\mK,\Delta}
\begin{bmatrix}
A & BC_{\mK} \\
B_{\mK} C & A_{\mK} 
\end{bmatrix}^\tr
+M_1(X_{\mK},\Delta) = 0,
\end{equation}
with
\begin{align*}
M_1(X_{\mK},\Delta)
\coloneqq
\begin{bmatrix}
0 & B\Delta_{C_\mK} \\
\Delta_{B_\mK} C & \Delta_{A_\mK}
\end{bmatrix} X_{\mK}
+ X_{\mK}\begin{bmatrix}
0 & B\Delta_{C_\mK} \\
\Delta_{B_\mK} C & \Delta_{A_\mK}
\end{bmatrix}^\tr 
\!+
\begin{bmatrix}
0 & 0 \\
0 & B_{\mK} V\Delta_{B_\mK}^\tr
\!+\! 
\Delta_{B_\mK} V {B_{\mK} }^\tr
\end{bmatrix}.
\end{align*}
\end{lemma}

We are now ready to prove the main technical result in \Cref{theorem:non_globally_optimal_stationary_point}. Our proof is motivated by that of~\cite[Theorem 4.2]{zheng2021analysis} with more complicated and careful computation. The main idea is to exploit the bilinear property of the Hessian and the non-controllable/non-observable property to identify a two-by-two hessian block
\begin{equation*} 
      \begin{bmatrix}
    \operatorname{Hess}_{\,\tilde{\mK}}(\Delta^{(1)},\Delta^{(1)}) & \operatorname{Hess}_{\,\tilde{\mK}}(\Delta^{(1)},\Delta^{(2)})\\
    \operatorname{Hess}_{\,\tilde{\mK}}(\Delta^{(1)},\Delta^{(2)}) & \operatorname{Hess}_{\,\tilde{\mK}}(\Delta^{(2)},\Delta^{(2)})
    \end{bmatrix} \in \mathbb{S}^2
\end{equation*}
in which the diagonal entries are always zero. Using the Hessian calculation in \Cref{lemma:Jn_Hessian}, we then prove that if $\operatorname{eig}(-\Lambda) \nsubseteq \mathcal{Z}$, then the off-diagonal entries are non-zero. The Hessian of $J_n(\mK)$ at $\tilde{\mK}$ is thus indefinite. 

\textbf{Proof of \Cref{theorem:non_globally_optimal_stationary_point}}: Consider a direction as 
$$
\Delta
    =\left[\begin{array}{c:cc}
    0 & 0 &  \Delta_C \\[2pt]
    \hdashline
    0 & 0 & 0 \\[-2pt]
    \Delta_B & 0 & \Delta_A
    \end{array}\right] \in \mathcal{V}_{n}, \quad \text{with} \; \Delta_A \in \mathbb{R}^{(n-q) \times (n-q)}, \Delta_B \in \mathbb{R}^{(n-q) \times p}, \Delta_C \in \mathbb{R}^{m\times (n-q) }.  
$$
The corresponding controller $\tilde{\mK} + \Delta$ in frequency domain is 
$$
\begin{aligned}
      \begin{bmatrix}
     \hat{C}_{\mK} & \Delta_C
    \end{bmatrix}\left(sI - \begin{bmatrix}\hat{A}_{\mK} & 0 \\ 0 & \Lambda + \Delta_A \end{bmatrix}\right)^{-1}\begin{bmatrix}
     \hat{B}_{\mK} \\ \Delta_D
    \end{bmatrix} 
    = \;& \begin{bmatrix}
     \hat{C}_{\mK} & \Delta_C
    \end{bmatrix}\begin{bmatrix}sI - \hat{A}_{\mK} & 0 \\ 0 & sI - \Lambda - \Delta_A  \end{bmatrix}^{-1}\begin{bmatrix}
     \hat{B}_{\mK} \\ \Delta_D
    \end{bmatrix} \\
    =\;&\hat{C}_{\mK}(sI - \hat{A}_{\mK})^{-1}\hat{B}_{\mK} + \Delta_C(sI - \Lambda - \Delta_A)^{-1}\Delta_B.
\end{aligned}
$$
We let 
$$
    \Delta^{(1)} = \left[\begin{array}{c:cc}
    0 & 0 &  \Delta_C \\[2pt]
    \hdashline
    0 & 0 & 0 \\[-2pt]
    0 & 0 & 0
    \end{array}\right], \Delta^{(2)} = \left[\begin{array}{c:cc}
    0 & 0 &  0 \\[2pt]
    \hdashline
    0 & 0 & 0 \\[-2pt]
    \Delta_B & 0 & 0
    \end{array}\right],\Delta^{(3)} = \left[\begin{array}{c:cc}
    0 & 0 &  0 \\[2pt]
    \hdashline
    0 & 0 & 0 \\[-2pt]
    0 & 0 & \Delta_A
    \end{array}\right].
$$
Now it is clear that the controllers $\tilde{\mK} + t \Delta^{(i)}, i = 1, 2, 3$ and $\tilde{\mK} + t (\Delta^{(i)} +\Delta^{(3)}), i = 1, 2$ correspond to the same transfer function in the frequency domain. Therefore, for all sufficiently small $t$, we have
$$
\begin{aligned}
J_n( \tilde{\mK}) &= J_n( \tilde{\mK} + t \Delta^{(1)}) = J_n( \tilde{\mK} + t \Delta^{(2)}) = J_n( \tilde{\mK} + t \Delta^{(3)})  \\
&=J_n( \tilde{\mK} + t (\Delta^{(1)}+\Delta^{(3)})) = J_n( \tilde{\mK} + t (\Delta^{(2)}+\Delta^{(3)})). 
\end{aligned}
$$
This implies that 
$$
\operatorname{Hess}_{\,\tilde{\mK}}
(\Delta^{(i)},\Delta^{(i)})=0,\qquad\forall i=1,2,3,
$$
and
$$
\operatorname{Hess}_{\,\tilde{\mK}}
(\Delta^{(1)}+\Delta^{(3)},\Delta^{(1)}+\Delta^{(3)})
=\operatorname{Hess}_{\,\tilde{\mK}}
(\Delta^{(2)}+\Delta^{(3)},\Delta^{(2)}+\Delta^{(3)})=0.
$$

Then, by the bilinearity of the Hessian, we have 
$$
\begin{aligned}
    \operatorname{Hess}_{\,\tilde{\mK}}(\Delta,\Delta) &=  \operatorname{Hess}_{\,\tilde{\mK}}(\Delta^{(1)}+\Delta^{(2)}+\Delta^{(3)},\Delta^{(1)}+\Delta^{(2)}+\Delta^{(3)}) \\
    & = \operatorname{Hess}_{\,\tilde{\mK}}(\Delta^{(1)}+\Delta^{(2)},\Delta^{(1)}+\Delta^{(2)}) + 2\operatorname{Hess}_{\,\tilde{\mK}}(\Delta^{(3)},\Delta^{(1)}+\Delta^{(2)}) + \operatorname{Hess}_{\,\tilde{\mK}}(\Delta^{(3)},\Delta^{(3)}) \\
    &= \operatorname{Hess}_{\,\tilde{\mK}}(\Delta^{(1)}+\Delta^{(2)},\Delta^{(1)}+\Delta^{(2)}). 
\end{aligned}
$$
We also have 
$$
\begin{aligned}
    \operatorname{Hess}_{\,\tilde{\mK}}(\Delta^{(1)},\Delta^{(2)}) &= \frac{1}{2} \operatorname{Hess}_{\,\tilde{\mK}}(\Delta^{(1)}+\Delta^{(2)},\Delta^{(1)}+\Delta^{(2)}) - \operatorname{Hess}_{\,\tilde{\mK}}(\Delta^{(1)},\Delta^{(1)}) - \operatorname{Hess}_{\,\tilde{\mK}}(\Delta^{(2)},\Delta^{(2)})\\
    &= \frac{1}{2} \operatorname{Hess}_{\,\tilde{\mK}}(\Delta^{(1)}+\Delta^{(2)},\Delta^{(1)}+\Delta^{(2)})\\
    &=\frac{1}{2} \operatorname{Hess}_{\,\tilde{\mK}}(\Delta,\Delta).
\end{aligned}
$$
If there exists a direction 
$
\Delta
    =\left[\begin{array}{c:cc}
    0 & 0 &  \Delta_C \\[2pt]
    \hdashline
    0 & 0 & 0 \\[-2pt]
    \Delta_B & 0 & \Delta_A
    \end{array}\right] \in \mathcal{V}_{n},
$
such that $\operatorname{Hess}_{\,\tilde{\mK}}(\Delta,\Delta) \neq 0$, then $\operatorname{Hess}_{\,\tilde{\mK}}$ must be indefinite  since there is a two-by-two block 
\begin{equation} \label{eq:indefiniteblock}
    \begin{bmatrix}
    \operatorname{Hess}_{\,\tilde{\mK}}(\Delta^{(1)},\Delta^{(1)}) & \operatorname{Hess}_{\,\tilde{\mK}}(\Delta^{(1)},\Delta^{(2)})\\
    \operatorname{Hess}_{\,\tilde{\mK}}(\Delta^{(1)},\Delta^{(2)}) & \operatorname{Hess}_{\,\tilde{\mK}}(\Delta^{(2)},\Delta^{(2)})
    \end{bmatrix}
\end{equation}
which has zero diagonal entries and non-zero off-diagonal entries.  

The rest of the proof is show that if $\Lambda$ is symmetric and $\operatorname{eig}(-\Lambda) \nsubseteq \mathcal{Z}$ with $\mathcal{Z}$ defined in \Cref{theorem:non_globally_optimal_stationary_point}, we can indeed find a direction $\Delta$ such that $\operatorname{Hess}_{\,\tilde{\mK}}(\Delta,\Delta) \neq 0$. 

\noindent \textbf{{Computation of $\operatorname{Hess}_{\,\tilde{\mK}}(\Delta,\Delta)$}:} Consider the controller $ \tilde{\mK}$ and a direction $\hat{\Delta}$ as 
$$
\begin{aligned}
   \tilde{\mK}
    &=\left[\begin{array}{c:cc}
    0 & \hat{C}_{\mK} &  0 \\[2pt]
    \hdashline
    \hat{B}_{\mK} & \hat{A}_{\mK} & 0 \\[-2pt]
    0 & 0 & \Lambda
    \end{array}\right] \in \mathcal{C}_{n}, \qquad 
    \hat{\Delta}
    &=\left[\begin{array}{c:cc}
    0 & 0 &  \Delta_C \\[2pt]
    \hdashline
    0 & 0 & 0 \\[-2pt]
    \Delta_B & 0 & 0
    \end{array}\right] \in \mathcal{V}_{n}.
\end{aligned}
$$
where $\Lambda$ is an $(n-q) \times (n-q)$ stable symmetric matrix. 

We let $X_{\tilde{\mK}}$ and $Y_{\tilde{\mK}}$ be the unique positive semidefinite solutions to the Lyapunov equations
\begin{subequations}
\begin{align}
\begin{bmatrix} A &  BC_{\tilde{\mK}} \\ B_{\tilde{\mK}} C & A_{\tilde{\mK}} \end{bmatrix}X_{\tilde{\mK}} + X_{\tilde{\mK}}\begin{bmatrix} A &  BC_{\tilde{\mK}} \\ B_{\tilde{\mK}} C & A_{\tilde{\mK}} \end{bmatrix}^\tr +  \begin{bmatrix} W & 0 \\ 0 & B_{\tilde{\mK}}VB_{\tilde{\mK}}^\tr  \end{bmatrix}
& = 0, \label{eq:LyapunovX}
\\
\begin{bmatrix} A &  BC_{\tilde{\mK}} \\ B_{\tilde{\mK}} C & A_{\tilde{\mK}} \end{bmatrix}^\tr Y_{\tilde{\mK}} +  Y_{\tilde{\mK}}\begin{bmatrix} A &  BC_{\tilde{\mK}} \\ B_{\tilde{\mK}} C & A_{\tilde{\mK}} \end{bmatrix} +   \begin{bmatrix} Q & 0 \\ 0 & C_{\tilde{\mK}}^\tr R C_{\tilde{\mK}} \end{bmatrix}
& = 0. \label{eq:LyapunovY}
\end{align}
\end{subequations}
Note that~\cref{eq:LyapunovX} reads as
\begin{equation} \label{eq:LyapunovX-partition}
\begin{bmatrix} A &  B\hat{C}_{\mK} &  0 \\  \hat{B}_{\mK} C & \hat{A}_{\mK} & 0 \\ 0 & 0 & \Lambda \end{bmatrix}X_{\tilde{\mK}} + X_{\tilde{\mK}}\begin{bmatrix} A &  B\hat{C}_{\mK} &  0 \\  \hat{B}_{\mK} C & \hat{A}_{\mK} & 0 \\ 0 & 0 & \Lambda \end{bmatrix}^\tr +  \begin{bmatrix} W & 0 & 0\\ 0 & \hat{B}_{\mK} V(\hat{B}_{\mK})^\tr & 0 \\ 0 & 0 & 0  \end{bmatrix}
 = 0.
\end{equation}
Then, the solutions $X_{\tilde{\mK}}$ and $Y_{\tilde{\mK}}$ to the Lyapunov equations~\cref{eq:LyapunovX} and~\cref{eq:LyapunovY} are in the form of 
$$
\begin{aligned}
    X_{\tilde{\mK}} &= \begin{bmatrix} X_{\operatorname{op}} & 0 \\ 0 & 0 \end{bmatrix} \in \mathbb{S}_+^{2n}, \qquad  \\
    Y_{\tilde{\mK}} &= \begin{bmatrix} Y_{\operatorname{op}} & 0 \\ 0 & 0 \end{bmatrix}  \in \mathbb{S}_+^{2n}
\end{aligned}
$$
where $X_{\operatorname{op}} \in \mathbb{S}^{n+q}_{+}$ is the unique positive semidefinite solution to the left-upper Lyapunov equation in~\cref{eq:LyapunovX-partition}, i.e.,
\begin{equation*} 
\begin{bmatrix} A &  B\hat{C}_{\mK}  \\  \hat{B}_{\mK} C & \hat{A}_{\mK} \end{bmatrix}X_{\operatorname{op}} + X_{\operatorname{op}}\begin{bmatrix} A &  B\hat{C}_{\mK}  \\  \hat{B}_{\mK} C & \hat{A}_{\mK} \end{bmatrix}^\tr +  \begin{bmatrix} W & 0 \\ 0 & \hat{B}_{\mK} V(\hat{B}_{\mK})^\tr   \end{bmatrix}
 = 0,
\end{equation*}
which is the same as~\cref{eq:LyapunovLQGX}, and $Y_{\operatorname{op}} \in \mathbb{S}^{n+q}_{+}$ is the unique positive semidefinite solution to~\cref{eq:LyapunovLQGY}. 

We next use \Cref{lemma:Jn_Hessian} to compute the hessian $\operatorname{Hess}_{\,\tilde{\mK}}(\hat{\Delta},\hat{\Delta})$. For this computation, we note that 
$$
\begin{aligned}
    \Delta_{B_\mK} &= \begin{bmatrix}
    0_{q \times p} \\   \Delta_{B}
    \end{bmatrix} \in \mathbb{R}^{n \times p}, \; \\
    \Delta_{C_\mK} &= \begin{bmatrix}
    0_{m \times q}  &  \Delta_{C}
    \end{bmatrix} \in \mathbb{R}^{m \times n}
\end{aligned}
$$
and that 
$$
\begin{aligned}
    \begin{bmatrix}
0 & 0 \\
0 & \Delta_{B_\mK}V\Delta_{B_\mK}^\tr
\end{bmatrix} Y_{\tilde{\mK}} = 
\left[\begin{array}{cc:c}
0 & 0 & 0\\
0 & 0 & 0 \\[2pt] 
\hdashline
0 & 0 & \Delta_{B}V\Delta_{B}^\tr
    \end{array}\right] 
\begin{bmatrix} Y_{\operatorname{op}} & 0 \\ 0 & 0 \end{bmatrix} &\equiv 0 \\
\begin{bmatrix}
0 & 0 \\ 0 & \Delta_{C_\mK}^\tr R \Delta_{C_\mK}
\end{bmatrix}\tilde{X}_{\mK} = \begin{bmatrix}
0 & 0 \\ 0 & \Delta_{C_\mK}^\tr R \Delta_{C_\mK}
\end{bmatrix}\begin{bmatrix} X_{\operatorname{op}} & 0 \\ 0 & 0 \end{bmatrix} &\equiv0.
\end{aligned}
$$
Therefore, by \Cref{lemma:Jn_Hessian}, we can see that 
\begin{equation} \label{eq:Hessian}
\operatorname{Hess}_{\, \tilde{\mK}}(\hat{\Delta},\hat{\Delta})
=
4\operatorname{tr}
\!\left(
\begin{bmatrix}
0 & B\Delta_{C_\mK} \\
\Delta_{B_\mK}C & 0
\end{bmatrix}
X'_{ \tilde{\mK},\hat{\Delta}}
\begin{bmatrix}
Y_{\mathrm{op}} & 0 \\ 0 & 0
\end{bmatrix}
\right) + 4\operatorname{tr}\!\left(\begin{bmatrix}
0 & 0 \\
0 & C_{ \tilde{\mK}}^\tr R \Delta_{C_\mK} 
\end{bmatrix}
X'_{ \tilde{\mK},\hat{\Delta}}\right),
\end{equation}
where $X'_{ \tilde{\mK},\hat{\Delta}}$ is the solution to the following Lyapunov equation
\begin{equation} \label{eq:hessianX-1}
\begin{aligned}
& \begin{bmatrix} A &  B\hat{C}_{\mK} &  0 \\  \hat{B}_{\mK} C & \hat{A}_{\mK}& 0 \\ 0 & 0 & \Lambda \end{bmatrix}X'_{ \tilde{\mK},\hat{\Delta}}
+
X'_{ \tilde{\mK},\Delta}
\begin{bmatrix} A &  B\hat{C}_{\mK} &  0 \\  \hat{B}_{\mK} C & \hat{A}_{\mK}& 0 \\ 0 & 0 & \Lambda \end{bmatrix}^\tr + M_1(X_{ \tilde{\mK}},\hat{\Delta})
=0,
\end{aligned}
\end{equation}
In~\cref{eq:hessianX-1}, the matrix $M_1(X_{ \tilde{\mK}},\hat{\Delta})$ is defined as 
$$
\begin{aligned}
&M_1(X_{ \tilde{\mK}},\hat{\Delta}) \\
=& 
\left[\begin{array}{cc:c}
0 & 0 & B\Delta_{C} \\
0 & 0 & 0 \\
\hdashline
\Delta_{B}C & 0 & 0
    \end{array}\right] \!
\begin{bmatrix}
X_{\mathrm{op}} & 0 \\
0 & 0
\end{bmatrix}
+
\begin{bmatrix}
X_{\mathrm{op}} & 0 \\
0 & 0
\end{bmatrix}
\left[\begin{array}{cc:c}
0 & 0 & B\Delta_{C} \\
0 & 0 & 0 \\
\hdashline
\Delta_{B}C & 0 & 0
    \end{array}\right] ^\tr \!\!+ 
\left[\begin{array}{cc:c}
0 & 0 & 0 \\
 0 & 0 & \hat{B}_{\mK} V\Delta_B^\tr \\
 \hdashline
0 &  \Delta_B V (\hat{B}_{\mK})^\tr & 0
    \end{array}\right] \\
                                 =& \begin{bmatrix}
                                  0 & X_{\mathrm{op}}(\Delta_{B}\bar{C})^\tr \\
                                  \Delta_{B}\bar{C}X_{\mathrm{op}} & 0
                                 \end{bmatrix} + \begin{bmatrix}
                                 0 &  \bar{B}_{\mK}V \Delta_{B}^\tr\\ \Delta_{B}V\bar{B}_{\mK}^\tr  &   0
                                 \end{bmatrix} \\
                                 =& \begin{bmatrix}
                                 0_{(n+q)\times(n+q)} &  (X_{\mathrm{op}}\bar{C}^\tr +  \bar{B}_{\mK}V) \Delta_{B}^\tr\\
                                 \Delta_{B}(\bar{C}X_{\mathrm{op}} +  V\bar{B}_{\mK}^\tr)  & 0_{(n-q) \times (n-q)}
                                 \end{bmatrix} \in \mathbb{S}^{2n},
\end{aligned}                               
$$
where we have used the definition in \cref{eq:augmented-matrices}, i.e., 
$$
\bar{C} = \begin{bmatrix}
C & 0
\end{bmatrix} \in \mathbb{R}^{p \times (n+q)},\;\bar{B} = \begin{bmatrix}
B \\ 0
\end{bmatrix}  \in \mathbb{R}^{(n+q) \times m}, \;\bar{B}_\mK = \begin{bmatrix}
0 \\ \hat{B}_{\mK}
\end{bmatrix}  \in \mathbb{R}^{(n+q) \times p}.
$$

For notational convenience, we define 
$$
A_{\mathrm{cl}} = \begin{bmatrix} A &  B\hat{C}_{\mK} \\  \hat{B}_{\mK} C & \hat{A}_{\mK}\end{bmatrix} \in \mathbb{R}^{(n+q) \times (n+q)}.
$$
Then, the solution to the Lyapunov equation~\cref{eq:hessianX-1} becomes
\begin{equation}
\begin{aligned}
    X'_{ \tilde{\mK},\Delta} &= \int_0^\infty \left(\begin{bmatrix}
    e^{A_{\mathrm{cl}}t} & 0 \\ 0 &  e^{\Lambda t} 
    \end{bmatrix} \begin{bmatrix}
                                 0_{(n+q)\times(n+q)} &  (X_{\mathrm{op}}\bar{C}^\tr +  \bar{B}_{\mK}V) \Delta_{B}^\tr\\
                                 \Delta_{B}(\bar{C}X_{\mathrm{op}} +  V\bar{B}_{\mK}^\tr)  & 0_{(n-q) \times (n-q)}
                                 \end{bmatrix}\begin{bmatrix}
    e^{A_{\mathrm{cl}}t} & 0 \\ 0 &  e^{\Lambda t} 
    \end{bmatrix}^\tr \right)dt \\
    &=\int_0^\infty \begin{bmatrix}
                     0_{(n+q)\times(n+q)} &  e^{A_{\mathrm{cl}} t}(X_{\mathrm{op}}\bar{C}^\tr +  \bar{B}_{\mK}V) \Delta_{B}^\tr e^{\Lambda^\tr t}\\
                    e^{\Lambda t} \Delta_{B}(\bar{C}X_{\mathrm{op}}  +  V\bar{B}_{\mK}^\tr) e^{A_{\mathrm{cl}}^\tr t} & 0_{(n-q) \times (n-q)}
                                 \end{bmatrix} dt. 
\end{aligned}
\end{equation}
Now we have
\begin{equation}  \label{eq:Hessian_part1}
    \begin{aligned}
        &\operatorname{tr}
\!\left(
\begin{bmatrix}
0_{(n+q)\times(n+q)} & \bar{B}\Delta_{C} \\
\Delta_{B}\bar{C} & 0
\end{bmatrix}
X'_{ \tilde{\mK},\hat{\Delta}}
\begin{bmatrix}
Y_{\mathrm{op}} & 0 \\ 0 & 0
\end{bmatrix}
\right) \\
=&\int_0^\infty  \operatorname{tr}
\!\left(\begin{bmatrix}
0& \bar{B}\Delta_{C} \\
\Delta_{B}\bar{C} & 0
\end{bmatrix} 
\begin{bmatrix}
0 &  e^{A_{\mathrm{cl}} t}(X_{\mathrm{op}}\bar{C}^\tr +  \bar{B}_{\mK}V) \Delta_{B}^\tr e^{\Lambda^\tr t}\\
e^{\Lambda t} \Delta_{B}(\bar{C}X_{\mathrm{op}}  +  V\bar{B}_{\mK}^\tr) e^{A_{\mathrm{cl}}^\tr t} & 0
\end{bmatrix}
\begin{bmatrix}
Y_{\mathrm{op}} & 0 \\ 0 & 0
\end{bmatrix}\right)dt \\
=&\int_0^\infty  \operatorname{tr}
\!\left(\bar{B}\Delta_{C}e^{\Lambda t} \Delta_{B}(\bar{C}X_{\mathrm{op}}  +  V\bar{B}_{\mK}^\tr) e^{A_{\mathrm{cl}}^\tr t} Y_{\mathrm{op}}  \right) 
dt. 
    \end{aligned}
\end{equation}

For the second part in~\cref{eq:Hessian}, we have 
$$
\begin{bmatrix}
0_{n\times n} & 0 \\
0 & C_{ \tilde{\mK}}^\tr R \Delta_{C_\mK} 
\end{bmatrix} = \begin{bmatrix}
0_{n\times n} & 0 \\
0 & \begin{bmatrix} \hat{C}_{\mK} & 0 \end{bmatrix}^\tr R \begin{bmatrix} 0 & \Delta_{C}  \end{bmatrix} 
\end{bmatrix} = \left[\begin{array}{cc:c}
0 & 0 & 0 \\
0 & 0 & (\hat{C}_{\mK})^\tr R \Delta_{C} \\
\hdashline
0 & 0 & 0
    \end{array}\right].
$$

Then, we have
\begin{equation} \label{eq:Hessian_part2}
\begin{aligned}
    &\operatorname{tr}\!\left(\left[\begin{array}{cc:c}
0 & 0 & 0 \\
0 & 0 & (\hat{C}_{\mK})^\tr R \Delta_{C} \\
\hdashline
0 & 0 & 0
    \end{array}\right]
X'_{ \tilde{\mK},\hat{\Delta}}\right) \\
=& \int_0^\infty \operatorname{tr}\!\left(\begin{bmatrix}
0_{(n+q)\times(n+q)} & \bar{C}_{\mK}^\tr {R} \Delta_{C}  \\
0 & 0
\end{bmatrix}\begin{bmatrix}
                     0_{(n+q)\times(n+q)} &  e^{A_{\mathrm{cl}} t}(X_{\mathrm{op}}\bar{C}^\tr +  \bar{B}_{\mK}V) \Delta_{B}^\tr e^{\Lambda^\tr t}\\
                   \star & 0_{(n-q) \times (n-q)}
                                 \end{bmatrix} \right)dt \\
                                 =& \int_0^\infty \operatorname{tr}\!\left( \bar{C}_{\mK}^\tr {R} \Delta_{C}e^{\Lambda t} \Delta_{B}(\bar{C}X_{\mathrm{op}}  +  V\bar{B}_{\mK}^\tr) e^{A_{\mathrm{cl}}^\tr t}\right)dt,
\end{aligned}
\end{equation}
where we have defined 
$$
\bar{C}_{\mK} = \begin{bmatrix}
0 & \hat{C}_{\mK}
\end{bmatrix} \in \mathbb{R}^{m \times (n+q)}
$$

Substituting~\cref{eq:Hessian_part1} and~\cref{eq:Hessian_part2} into~\cref{eq:Hessian} leads to
\begin{equation}
\begin{aligned}
    \operatorname{Hess}_{\, \tilde{\mK}}(\hat{\Delta},\hat{\Delta})
=&
4\int_0^\infty  \operatorname{tr}
\!\left(\bar{B}\Delta_{C}e^{\Lambda t} \Delta_{B}(\bar{C}X_{\mathrm{op}}  +  V\bar{B}_{\mK}^\tr) e^{A_{\mathrm{cl}}^\tr t} Y_{\mathrm{op}} \right)dt \\
&+4\int_0^\infty \operatorname{tr}\!\left( \bar{C}_{\mK}^\tr  {R} \Delta_{C}e^{\Lambda t} \Delta_{B}(\bar{C}X_{\mathrm{op}}  +  V\bar{B}_{\mK}^\tr) e^{A_{\mathrm{cl}}^\tr t}\right)dt \\
= &4 \int_0^\infty \operatorname{tr}\!\left( \Delta_{C}e^{\Lambda t} \Delta_{B}(\bar{C}X_{\mathrm{op}}  +  V\bar{B}_{\mK}^\tr) e^{A_{\mathrm{cl}}^\tr t} (Y_{\mathrm{op}}\bar{B} +  \bar{C}_{\mK}^\tr R ) \right) dt
\end{aligned} 
\end{equation}


Let $\lambda\in \operatorname{eig}(-\Lambda)\backslash\mathcal{Z}$. Since $\lambda\notin \mathcal{Z}$, there exists some $i,j$ such that
\begin{equation} \label{eq:Glambda}
G(\lambda)
\coloneqq 
{e_i^{(p)}}^\tr (\bar{C}X_{\mathrm{op}}  +  V\bar{B}_{\mK}^\tr)(\lambda I - A_{_{\mathrm{cl}}}^\tr )^{-1} (Y_{\mathrm{op}}\bar{B} +  \bar{C}_{\mK}^\tr R )) e_j^{(m)}
\neq 0,
\end{equation}
where $e_i^{(p)} \in \mathbb{R}^{p}$ is the $i$-th element of the standard basis in $\mathbb{R}^{p}$ and  $e_j^{(m)} \in \mathbb{R}^{m}$  the $j$-th element of the standard basis in $\mathbb{R}^{m}$. 

We consider a symmetric $\Lambda \in \mathbb{S}^{n-q}$, then it can be diagonalized with real eigenvalues
$$
    T \Lambda T^{-1} = \begin{bmatrix}
    -\lambda & 0 \\  0 & *
    \end{bmatrix}  \in \mathbb{S}^{n-q}.
$$
We then choose a special direction 
$$\hat{\Delta}
    =\left[\begin{array}{c:cc}
    0 & 0 &  \Delta_C \\[2pt]
    \hdashline
    0 & 0 & 0 \\[-2pt]
    \Delta_B & 0 & 0
    \end{array}\right] \in \mathcal{V}_{n},
    $$
with 
$$
    \Delta_C = e_j^{(m)}\left(e_1^{(n-q)}\right)^\tr T^{-1}, \quad  \Delta_B = Te_1^{(n-q)}\left(e_i^{(p)}\right)^\tr
$$
For this choice, we have 
$$
    \Delta_{C}e^{\Lambda t} \Delta_{B} = e^{-\lambda t}  e_j^{(m)} \left(e_i^{(p)}\right)^\tr,
$$
leading to 
\begin{equation}
\begin{aligned}
    \operatorname{Hess}_{\, \tilde{\mK}}(\hat{\Delta},\hat{\Delta})
&=
4\int_0^\infty  e_i^\tr (\bar{C}X_{\mathrm{op}}  +  V\bar{B}_{\mK}^\tr)(\lambda I - A_{_{\mathrm{cl}}}^\tr )^{-1} (Y_{\mathrm{op}}\bar{B} +  \bar{C}_{\mK}^\tr R )) e_jdt \\
&=   4 \left(e_i^{(p)}\right)^\tr (\bar{C}X_{\mathrm{op}}  +  V\bar{B}_{\mK}^\tr)(\lambda I - A_{_{\mathrm{cl}}}^\tr )^{-1} (Y_{\mathrm{op}}\bar{B} +  \bar{C}_{\mK}^\tr R ))e_j^{(m)},\\
& = 4G(\lambda),
\end{aligned}
\end{equation}
which is not zero by~\cref{eq:Glambda}. We thus have an indefinite $2$-by-$2$ hessian block in~\cref{eq:indefiniteblock}.  
This completes the proof. \hfill $\square$

\subsection{Hessian Computations in \Cref{example:doyle,example:non-minimal,example:open-loop-stable}}

Here, we present some details about the Hessian calculations in \Cref{example:doyle,example:non-minimal,example:open-loop-stable}. For all the Hessian calculation below, we use the standard basis in $\mathcal{V}_n$. 

\begin{enumerate}
    \item \Cref{example:doyle}: 
The Hessian of $J_2(\mK)$ at the optimal controller 
${\mK}^\star = 
= \left[\begin{array}{c:cc}
0 &   -5 & -5 \\[2pt]
    \hdashline 5 & -4 &    1 \\
 5& -10 &-4
\end{array}\right] \in \mathcal{C}_{2}
$
is
$$
   \operatorname{Hess}_{\,{\mK}^\star} = \begin{bmatrix}
    1.1321  & -0.9375  &  0.9375   & 1.0889 &  -0.9056 &  -1.1321  & -1.2916  &  1.0889 \\
   -0.9375  &  0.7871  & -0.7871  & -0.8980 &   0.7571  &  0.9375 &   1.0517 &  -0.8980\\
    0.9375  & -0.7871  &  0.7871  &  0.8980 &  -0.7571  & -0.9375 &  -1.0517 &   0.8980 \\
    1.0889  & -0.8980  &  0.8980  &  1.0488 &  -0.8685  & -1.0889  & -1.2492  &  1.0488 \\
   -0.9056  &  0.7571 &  -0.7571&   -0.8685   & 0.7293  &  0.9056 &   1.0208 &  -0.8685\\
   -1.1321  &  0.9375 &  -0.9375  & -1.0889  &  0.9056   & 1.1321  &  1.2916  & -1.0889\\
   -1.2916  &  1.0517 &  -1.0517  & -1.2492  &  1.0208  &  1.2916  &  1.5084 &  -1.2492 \\
    1.0889 &  -0.8980  &  0.8980  &  1.0488  & -0.8685 &  -1.0889  & -1.2492  &  1.0488
\end{bmatrix} \times 10^5
$$
which is positive semidefinite and has eigenvalues 
$$
\lambda_1 =     8.1111 \times 10^5,  \lambda_2 =   6\,133.9, \lambda_3 =    131.2,  \lambda_4 =     6.36,  \lambda_5 = \cdots = \lambda_8 = 0.
$$
We further compute the matrices in~\cref{eq:augmented-matrices} as follows
 $$\bar{C} = \begin{bmatrix}
1 & 0 &0 & 0
\end{bmatrix}, \;
\bar{B} = \begin{bmatrix}
0\\1 \\ 0\\0
\end{bmatrix},\; \bar{C}_{\mK} = \begin{bmatrix}
0 & 0 &   -5 & -5
\end{bmatrix},\; \bar{B}_\mK = \begin{bmatrix}
0 \\ 0\\ 5 \\ 5
\end{bmatrix}, \; A_{\mathrm{cl}} \!=\! \begin{bmatrix}
    1  &  1    &     0     &    0\\
         0   & 1  & -5 &  -5\\
    5     &    0  & -4 &   1\\
    5     &    0  &-10&   -4
\end{bmatrix}
$$
and the unique solutions to Lyapunov equations~\cref{eq:LyapunovLQGX} and~\cref{eq:LyapunovLQGY} are 
$$
\begin{aligned}
  X_{\mathrm{op}} &\!= \!\frac{1}{6}\begin{bmatrix}
  680  &-695   &650 & -725\\
 -695  & 985 & -725  & 925\\
  650  &-725  & 650 & -725\\
 -725  & 925  &-725  & 925
\end{bmatrix}, \\ 
Y_{\mathrm{op}} &\!=\! \frac{1}{6}\begin{bmatrix}
  985  & -695&  -925 &  725\\
 -695 &  680  & 725 & -650\\
 -925 &  725 &  925 & -725 \\
  725 & -650 & -725 &  650
\end{bmatrix}.
\end{aligned}
$$
Then, we can compute 
$$
    \begin{aligned}
    (\bar{C}X_{\mathrm{op}}  +  V\bar{B}_{\mK}^\tr)(s I - A_{_{\mathrm{cl}}}^\tr )^{-1}Y_{\mathrm{op}} \bar{B} 
    = &\frac{-12.5 s^3 - 604.2 s^2 - 1712 s - 566.7}{    s^4 + 6 s^3 + 11 s^2 + 6 s + 1},\\ 
        (\bar{C}X_{\mathrm{op}}  +  V\bar{B}_{\mK}^\tr)(s I - A_{_{\mathrm{cl}}}^\tr )^{-1} \bar{C}_{\mK}^\tr R
        =& \frac{  12.5 s^3 + 604.2 s^2 + 1713 s + 566.7}{  s^4 + 6 s^3 + 11 s^2 + 6 s + 1}. 
    \end{aligned}
    $$
    Thus, we have 
    $
    \mathbf{G}(s) \!=\! (\bar{C}X_{\mathrm{op}}  +  V\bar{B}_{\mK}^\tr)(s I - A_{_{\mathrm{cl}}}^\tr )^{-1} (Y_{\mathrm{op}}\bar{B} +  \bar{C}_{\mK}^\tr R )) \equiv 0.
    $
This result that $\mathbf{G}(s)$ being identically zero is expected from~\Cref{theorem:non_globally_optimal_stationary_point} since $\mK^\star$ is globally optimal.

\item \Cref{example:non-minimal}: In this example, the globally optimal controller from Riccati equation is non-minimal, at which the Hessian is 
$$
   \operatorname{Hess}_{\,{\mK}^\star} = \begin{bmatrix}
      230 &   0  & -115 &   73.5&     0&   208 &  -136.5 &    0\\
    0 &    0  &  0    & 0 &    0 &   0    & 0 &    0\\
 -115 &    0  &  57.5 &   -36.75 &    0&  -104 &    68.25 &     0 \\
   73.5 &     0 &  -36.75 &    24.5&     0 &   63.75&   -44& 0\\
    0    & 0  &  0&     0 &    0&    0&     0 &    0\\
  208 &    0&  -104&    63.75&    0 &  195.5&  -122.25&     0\\
 -136.5&     0&    68.25&   -44&     0&  -122.25&    81.5&    0\\
   0    & 0 &    0 &   0&     0&     0 &   0&    0
    \end{bmatrix}.
$$
This hessian is positive semidefinite and has eigenvalues
$$
\lambda_1 =  581.5529, \lambda_2 =  7.1879, \lambda_3 =  0.2592, \lambda_4 = \cdots = \lambda_8 = 0.
$$

We have shown the following two non-minimal full-order controllers
$$
\tilde{\mK}_1
    =\left[\begin{array}{c:cc}
    0 & -2 &  0 \\[2pt]
    \hdashline
    1 & -3 & 0 \\[-2pt]
    0 & 0 & \Lambda
    \end{array}\right], \quad \tilde{\mK}_2
    =\left[\begin{array}{c:cc}
    0 & 0.5 &  0 \\[2pt]
    \hdashline
    -4 & -3 & 0 \\[-2pt]
    0 & 0 & \Lambda
    \end{array}\right],  
$$
are also globally optimal to the LQG problem in \Cref{example:non-minimal}. The Hessian at these two controllers are
$$
\begin{aligned}
       \operatorname{Hess}_{\,\tilde{\mK}_1} &= \begin{bmatrix}
   230      &   0 &-115& 73.5&        0  &       0   &      0    &     0 \\
       0     &    0    &     0 &        0    &     0  &0&  0&         0\\
 -115&         0&  57.5& -36.75&       0   &      0   &      0  &       0\\
  73.5&        0 & -36.75&   24.5&       0     &    0    &     0   &      0\\
         0 &        0   &      0       &  0      &   0  &  0&   0&         0\\
         0  & 0&         0  &       0 &   0&         0 &        0 &        0\\  
         0  & 0&         0  &       0 &   0&         0 &        0 &        0\\  
         0  & 0&         0  &       0 &   0&         0 &        0 &        0
    \end{bmatrix},
\end{aligned}
$$
with eigenvalues $ \lambda_1 = 311.0647, \lambda_2 = 0.9353, \lambda_i = 0, i =3, \ldots 8 $
and 
$$
\begin{aligned}
       \operatorname{Hess}_{\,\tilde{\mK}_2} &= \begin{bmatrix}
   14.375      &   0 &-115& -18.375&        0  &       0   &      0    &     0 \\
       0     &    0    &     0 &        0    &     0  &0&  0&         0\\
 -115&         0&  920& 147&       0   &      0   &      0  &       0\\
  -18.375&        0 & 147&   24.5&       0     &    0    &     0   &      0\\
         0 &        0   &      0       &  0      &   0  &  0&   0&         0\\
         0  & 0&         0  &       0 &   0&         0 &        0 &        0\\  
         0  & 0&         0  &       0 &   0&         0 &        0 &        0\\  
         0  & 0&         0  &       0 &   0&         0 &        0 &        0
    \end{bmatrix},
\end{aligned}
$$
with eigenvalues  $ \lambda_1 = 957.8879, \lambda_2 =  0.9871, \lambda_i = 0, i =3, \ldots 8$.   
  We further compute the matrices in~\cref{eq:augmented-matrices} as follows
 $$\bar{C} = \begin{bmatrix}
1 & -1 &0 
\end{bmatrix}, \;
\bar{B} = \begin{bmatrix}
1\\0 \\ 0
\end{bmatrix},\; \bar{C}_{\mK} = \begin{bmatrix}
0 & 0 &   0.5
\end{bmatrix},\; \bar{B}_\mK = \begin{bmatrix}
0 \\ 0\\ -4
\end{bmatrix}, \; A_{\mathrm{cl}} \!=\! \begin{bmatrix}
    0 & -1 & 0.5 \\
    1 & 0 & 0 \\
    -4 & 4 & -3
\end{bmatrix}
$$
and the unique solutions to Lyapunov equations~\cref{eq:LyapunovLQGX} and~\cref{eq:LyapunovLQGY} are 
$$
\begin{aligned}
  X_{\mathrm{op}} &\!= \!\frac{1}{4}\begin{bmatrix}
   21 & -32 & -68\\
  -32  & 81 & 128 \\
  -68 & 128 & 272
\end{bmatrix}, \qquad 
Y_{\mathrm{op}} &\!=\! \frac{1}{8}\begin{bmatrix}
   32 &   0 &   4 \\
    0  & 16   &0\\
    4 &  0  &  1
\end{bmatrix}.
\end{aligned}
$$
Then, we can compute 
 $$
    \begin{aligned}
    (\bar{C}X_{\mathrm{op}}  +  V\bar{B}_{\mK}^\tr)(s I - A_{_{\mathrm{cl}}}^\tr )^{-1}Y_{\mathrm{op}} \bar{B}&= \frac{26.5s +   56.5}{(s + 1)^2} \\
        (\bar{C}X_{\mathrm{op}}  +  V\bar{B}_{\mK}^\tr)(s I - A_{_{\mathrm{cl}}}^\tr )^{-1} \bar{C}_{\mK}^\tr R&= -\frac{26.5s +   56.5}{(s + 1)^2}. 
    \end{aligned}
    $$
    Thus, we have 
    $
    \mathbf{G}(s) \!=\! (\bar{C}X_{\mathrm{op}}  +  V\bar{B}_{\mK}^\tr)(s I - A_{_{\mathrm{cl}}}^\tr )^{-1} (Y_{\mathrm{op}}\bar{B} +  \bar{C}_{\mK}^\tr R )) \equiv 0.
    $
This result that $\mathbf{G}(s)$ being identically zero is expected from~\Cref{theorem:non_globally_optimal_stationary_point} since $\tilde{\mK}_2$ is globally optimal.  

\item \Cref{example:open-loop-stable}: The system is open-loop stable. Therefore, according to \cite[Theorem 4.2]{zheng2021analysis}, the zero controller 
$\tilde{\mK} = \left[\begin{array}{c:c}
0 & 0 \\[2pt]
    \hdashline 0 & \Lambda
\end{array}\right]
\in \mathcal{C}_{2}$ with any stable $ \Lambda \in \mathbb{R}^{2 \times 2}$ is a stationary point. We compute the matrices in~\cref{eq:augmented-matrices} as
 $$\bar{C} = \begin{bmatrix}
-\frac{1}{6} & \frac{11}{12}
\end{bmatrix}, \;
\bar{B} = \begin{bmatrix}
 -1 \\
     1
\end{bmatrix},\; \bar{C}_{\mK} = \begin{bmatrix}
0 & 0 
\end{bmatrix},\; \bar{B}_\mK = \begin{bmatrix}
0 \\ 0
\end{bmatrix}, \; A_{\mathrm{cl}} \!=\! \begin{bmatrix}
       -0.5 &       0 \\
    0.5&  -1
\end{bmatrix}
$$
and the unique solutions to~\cref{eq:LyapunovLQGX} and~\cref{eq:LyapunovLQGY} are 
$
\begin{aligned}
  X_{\mathrm{op}} &\!= \!\frac{1}{3}\begin{bmatrix}
     3  &   1 \\
     1   &  2
\end{bmatrix}, \; 
Y_{\mathrm{op}} &\!=\! \frac{1}{6}\begin{bmatrix}
    7  &  1 \\
    1  &  3
\end{bmatrix}.
\end{aligned}
$
Then, we can compute 
 $$
    \mathbf{G}(s) \!=\! (\bar{C}X_{\mathrm{op}}  +  V\bar{B}_{\mK}^\tr)(s I - A_{_{\mathrm{cl}}}^\tr )^{-1} (Y_{\mathrm{op}}\bar{B} +  \bar{C}_{\mK}^\tr R )) = \frac{5(2s-1)}{108(2s^2 + 3s +1)}.
    $$
The zero set $\mathcal{Z} = \{0.5\}$ contains a single value. For any stable $\Lambda$ with $\operatorname{eig}(-\Lambda) \nsubseteq \mathcal{Z}$, the Hessian is indefinite by \Cref{theorem:non_globally_optimal_stationary_point}. For instance,  with $\Lambda = - \mathrm{diag}(0.5,0.1)$, the Hessian is 
$$
\begin{aligned}
       \operatorname{Hess}_{\,\tilde{\mK}} &= \begin{bmatrix}
         0    &     0  &  0 &        0    &     0    &     0  &       0    &     0 \\
         0   &      0  &       0     &    0  &       0 &  -0.0561      &   0     &    0 \\
     0    &     0  &  0 &        0    &     0    &     0  &       0    &     0 \\
        0    &     0  &  0 &        0    &     0    &     0  &       0    &     0 \\
         0    &     0  &  0 &        0    &     0    &     0  &       0    &     0 \\
         0  & -0.0561 &        0&         0     &    0 &        0  &       0    &     0 \\
         0    &     0  &  0 &        0    &     0    &     0  &       0    &     0 \\
         0    &     0  &  0 &        0    &     0    &     0  &       0    &     0 
    \end{bmatrix},
\end{aligned}
$$

which is indefinite with eigenvalues $\lambda_1 = 0.0561,\lambda_2 = -0.0561, \lambda_i = 0, i = 3, \ldots, 8$.
 However, we can check that if $\Lambda = -0.5\Ib_2$, (i.e. $
     A_{\mK} = -0.5 \Ib_2,\ B_{\mK} = 0,\ C_{\mK} = 0
$), its Hessian is degenerated to zero, implying that it is a high-order saddle.
  
\end{enumerate}
\end{document}